\documentclass[11pt]{amsart}
\usepackage[letterpaper, total={6.5in,9in}, includefoot, centering]{geometry}
\usepackage{amsfonts}
\usepackage{amsmath,amsthm,amssymb,amscd}
\usepackage[foot]{amsaddr}

\usepackage[mathscr]{eucal}
\usepackage{latexsym}
\usepackage[new]{old-arrows}
\usepackage{graphics}
\usepackage{verbatim}
\usepackage{upgreek}
\usepackage{overpic}
\usepackage[all,cmtip]{xy} 
\usepackage{enumitem}
\usepackage{tikz}
\usepackage{tikz-cd}
\usepackage[pagebackref]{hyperref}
\renewcommand*\backref[1]{\ifx#1\relax \else (Page #1) \fi}
\usepackage[font=scriptsize]{caption}
\usepackage{upgreek}

\usepackage[authoryear,sort&compress]{natbib}
\bibpunct{[}{]}{;}{n}{,}{,}

\makeatletter \@addtoreset{equation}{section}

\DeclareMathOperator{\ext}{ext}

\makeatother
\newtheorem{theorem}{Theorem}[section]

\newtheorem{corollary}{Corollary}[section]
\newtheorem{definition}{Definition}[section]
\newtheorem{remark}{Remark}[section]
\newtheorem{example}{Example}[section]
\newtheorem{prop}{Proposition}[section]
\newtheorem{assumption}{Assumption}[section]

\begin{document}
\title[Variational Structures in Cochain Projection Based Variational Discretizations]{Variational Structures in Cochain Projection Based Variational Discretizations of Lagrangian PDEs}
\author{Brian Tran and Melvin Leok}
\address{Department of Mathematics, University of California, San Diego, 9500 Gilman Drive, La Jolla, CA 92093-0112, USA.}
\email{b3tran@ucsd.edu, mleok@ucsd.edu}
\allowdisplaybreaks

\begin{abstract}
Compatible discretizations, such as finite element exterior calculus, provide a discretization framework that respect the cohomological structure of the de Rham complex, which can be used to systematically construct stable mixed finite element methods. Multisymplectic variational integrators are a class of geometric numerical integrators for Lagrangian and Hamiltonian field theories, and they yield methods that preserve the multisymplectic structure and momentum-conservation properties of the continuous system. In this paper, we investigate the synthesis of these two approaches, by constructing discretization of the variational principle for Lagrangian field theories utilizing structure-preserving finite element projections. In our investigation, compatible discretization by cochain projections plays a pivotal role in the preservation of the variational structure at the discrete level, allowing the discrete variational structure to essentially be the restriction of the continuum variational structure to a finite-dimensional subspace. The preservation of the variational structure at the discrete level will allow us to construct a discrete Cartan form, which encodes the variational structure of the discrete theory, and subsequently, we utilize the discrete Cartan form to naturally state discrete analogues of Noether's theorem and multisymplecticity, which generalize those introduced in the discrete Lagrangian variational framework by \citet{MaPaSh1998}. We will study both covariant spacetime discretization and canonical spatial semi-discretization, and subsequently relate the two in the case of spacetime tensor product finite element spaces.
\end{abstract}

\maketitle

\tableofcontents

\section{Introduction}
The problem of structure-preservation in numerical discretizations of partial differential equations has primarily been studied in two disjoint stages, the first involving the semi-discretization of the spatial degrees of freedom, and the second having to do with the time-integration of the resulting coupled system of ordinary differential equations. Implicit in such an approach is the use of tensor product meshes in spacetime. In the context of spatial semi-discretization, the notion of structure-preservation is focused on compatible discretizations~(see~\citet{Ar2018}, and references therein), that preserve in some manner the functional and geometric relationships between the different function spaces that arise in the partial differential equation, and in the context of time-integration, geometric numerical integrators~(see~\citet{HaLuWa2006}, and references therein) aim to preserve geometric invariants like the symplectic or Poisson structure, energy, momentum, and the nonlinear manifold structure of the configuration spaces, like its Lie group, homogeneous space, or Riemannian structure.

Lagrangian partial differential equations are an important class of partial differential equations that exhibit geometric structure, and they can benefit from numerical discretizations that preserve such geometric structure. This can either be viewed as an infinite-dimensional Lagrangian system with time as the independent variable, or a finite-dimensional Lagrangian multisymplectic field theory~\cite{MaPeShWe2001} with space and time as independent variables. Lagrangian variational integrators~\cite{MaWe2001, MaPaSh1998} are a popular method for systematically constructing symplectic integrators of arbitrarily high-order, and satisfy a discrete Noether's theorem that relates group-invariance with momentum conservation. A group-invariant (and hence momentum-preserving) variational integrator can be constructed from group-equivariant interpolation spaces~\cite{GaLe2018}.

In this paper, we will demonstrate how compatible discretization, multisymplectic variational integrators, and group-equivariant interpolation spaces can be combined to yield a natural geometric structure-preserving discretization framework for Lagrangian field theories. 

\paragraph{\bf Multisymplectic Formulation of Classical Field Theories} The variational principle for Lagrangian PDEs involve a multisymplectic formulation \cite{MaPaSh1998, MaPeShWe2001}. The base space $X$ consists of independent variables, denoted by $(x^0,\ldots,x^n)\equiv(t,x)$, where $x^0\equiv t$ is time, and $(x^1,\ldots,x^n)\equiv x$ are space variables. The dependent field variables, $(y^1,\ldots, y^m)\equiv y$, form a fiber over each spacetime basepoint. The independent and field variables form the configuration bundle, $\pi:Y\rightarrow X$. The configuration of the system is specified by a section of $Y$ over $X$, which is a continuous map $\phi:X\rightarrow Y$, such that $\pi\circ\phi=1_{X}$. This means that for every $(t,x)\in X$, $\phi((t,x))$ is in the fiber $\pi^{-1}((t,x))$ over $(t,x)$.

\begin{figure}[h]
\centerline{
\begin{overpic}
[scale=0.925]
{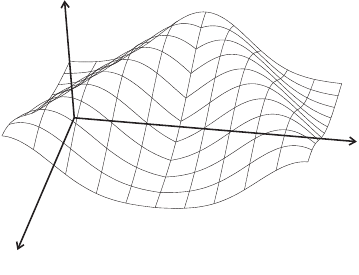}
\put(10,5){$X$}
\put(95,30){$t$}
\put(10,70){$x$}
\put(85,60){$\phi$}
\end{overpic}}
\vspace*{-0.15in}
  \caption{\footnotesize A section of the configuration bundle: the horizontal axes represent spacetime, and the vertical axis represent dependent field variables. The section $\phi$ gives the value of the field variables at every point of spacetime.}
  \label{fig:section_configuration_bundle}
\end{figure}
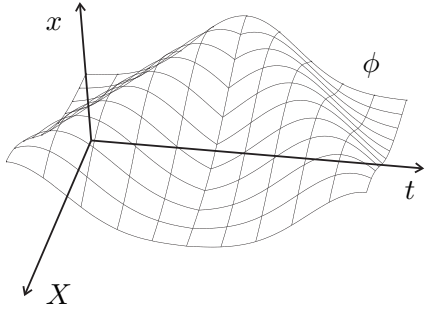

For ODEs, the Lagrangian depends on position and its time derivative, which is an element of the tangent bundle $TQ$, and the action is obtained by integrating the Lagrangian in time. In the multisymplectic case, the Lagrangian density is dependent on the field variables and the partial derivatives of the field variables with respect to the spacetime variables, and the action integral is obtained by integrating the Lagrangian density over a region of spacetime. The multisymplectic analogue of the tangent bundle is the first jet bundle $J^1 Y$, consisting of the configuration bundle $Y$, and the first partial derivatives of the field variables with respect to the independent variables. In coordinates, we have $\phi(x^0,\ldots, x^n)=(x^0,\ldots x^n, y^1,\ldots y^m)$, which allows us to denote the partial derivatives by
$ v^a_\mu={y^a}_{,\mu}=  \partial y^a / \partial x^\mu$.
We can think of $J^1 Y$ as a fiber bundle over $X$. Given a section $\phi:
X\rightarrow Y$, we obtain its first jet extension, $j^1 \phi:X\rightarrow J^1 Y$, that is given by
\[ j^1 \phi (x^0,\ldots, x^n) = \left( x^0,\ldots, x^n, y^1,\ldots, y^m, {y^1}_{,0},\ldots, {y^m}_{,n} \right),\]
which is a section of the fiber bundle $J^1 Y$ over $X$. We refer to sections of $J^1Y$ of the form $j^1\phi$, where $\phi$ is a section of $Y$, as holonomic. The configuration space is the space of sections of $Y$ and the velocity phase space is the space of holonomic sections of $J^1Y$. The Lagrangian density is a bundle map $\mathcal{L}:J^1 Y\rightarrow \wedge^{n+1}(T^*X)$ and hence, induces a map on the space of sections $\mathcal{L}: \Gamma(J^1Y) \rightarrow \Omega^{n+1}(X)$. Thus, we can define the action functional $S: \Gamma(Y) \rightarrow \mathbb{R}$ by
$S[\phi] = \int_{X} \mathcal{L}(j^1 \phi). $
Hamilton's principle states that $\delta S = 0$, subject to compactly supported variations. As we will see, this is the basis of Lagrangian multisymplectic variational integrators~\cite{MaPaSh1998}.

The variational structure of a Lagrangian field theory is given by the Cartan form, which in coordinates has the expression
\begin{equation}\label{Cartan form}
\Theta_{\mathcal{L}} = \frac{\partial L}{\partial v^a_\mu} dy^a \wedge d^nx_\mu + \left(L - \frac{\partial L}{\partial v^a_\mu}v^a_\mu\right) d^{n+1}x.
\end{equation}
This can be defined intrinsically as the pullback of the canonical $(n+1)$-form on the dual jet bundle by the covariant Legendre transform $\mathbb{F}\mathcal{L}: J^1Y \rightarrow J^1Y^*$. Then, the action can be expressed as $S[\phi] = \int_X\mathcal{L}(j^1\phi) = \int_X (j^1\phi)^*\Theta_{\mathcal{L}}.$ The variation of the action is then expressed as
$$ dS[\phi]\cdot V = -\int_X (j^1\phi)^*(j^1V \lrcorner\ \Omega_{\mathcal{L}}) + \int_{\partial X} (j^1\phi)^* (j^1V \lrcorner\ \Theta_{\mathcal{L}}), $$
where $\Omega_{\mathcal{L}} = -d\Theta_{\mathcal{L}}$ defines the multisymplectic form and $j^1V$ denotes the jet prolongation of the vector field $V$ (for details, see \citet{GoIsMaMo1998}). Hence, the variation of the action is completely specified by the Cartan form; we will show that a finite element discretization of the variational principle gives rise to a discrete form and subsequently we will express variational properties of the discrete system in terms of the discrete Cartan form.

In this paper, we will take the fields to be elements of $H\Lambda^k(X)$, the space of square integrable $k-$forms on $X$ with square integrable exterior derivative. In this setting, the appropriate analogue of the configuration space is $H\Lambda^k$ and the appropriate analogue of the velocity phase space is $J^1_{H\Lambda^k} := H\Lambda^k \times dH\Lambda^k$, where the jet extension of a field $\phi \in H\Lambda^k$, only depending on the exterior derivative, is $j^1_d\phi \equiv (x,\phi,d\phi)$, i.e., we consider Lagrangian theories that depend on the exterior derivative of the field and not depending more generally on all first-order derivatives; for scalar fields, $k=0$, these are equivalent.  We refer to $j^1_d: H\Lambda^k \rightarrow J^1_{H\Lambda^k}$ as the exterior jet extension.

\paragraph{\bf Finite Element Exterior Calculus} The notion of compatible discretization is a research area that has garnered significant interest and activity in the finite element community, motivated by the seminal work of \citet{ArFaWi2006} on finite element exterior calculus that provides a broad generalization of Hiptmair's work on mixed finite elements for electromagnetism \citep{Hi2002}. This arises from the fundamental role that the de~Rham complex of exterior differential forms plays in mixed formulations of elliptic partial differential equations, and the realization that many of the most successful mixed finite element spaces, such as Raviart--Thomas and N\'ed\'elec elements, can be viewed as finite element subspaces of the de~Rham complex that satisfy a bounded cochain projection property, so that the set of mixed finite elements form a subcomplex that provides stable approximations of the original problem.

\paragraph{\bf Group-equivariant interpolation}

The study of group-equivariant approximation spaces~\cite{GaLe2018} for functions taking values on manifolds is motivated by the applications to geometric structure-preserving discretization of Lagrangian and Hamiltonian PDEs with symmetries. In particular, when the Lagrangian density for a Lagrangian PDE with symmetry is discretized using a Lagrangian multisymplectic variational integrator constructed from an approximation space that is equivariant with respect to the symmetry group, the resulting numerical method automatically preserves the momentum map associated with the symmetry of the PDE. In essence, such variational discretizations exhibit a discrete analogue of Noether's theorem, which connects symmetries of the Lagrangian with momentum conservation laws.

Many intrinsic geometric flows such as the Ricci flow and the Einstein equations involves computing the evolution of a Riemannian or pseudo-Riemannian metric on spacetime. Additionally, these intrinsic geometric flows can often be formulated variationally, so it is natural to consider group-equivariant approximation spaces taking values on Riemannian or pseudo-Riemannian metrics with a view towards constructing variational discretizations that preserve the associated momentum maps.

A now standard approach to constructing an approximation space for functions taking values on a Riemannian manifold that is equivariant with respect to Riemannian isometries is the method of geodesic finite elements introduced independently by~\citet{Sa2012} and~\citet{Gr2013}. Given a Riemannian manifold $(M,g)$, the geodesic finite element $\varphi:\Delta^n\rightarrow M$ associated with a set of linear space finite elements $\{ v_i:\Delta^n\rightarrow\mathbb{R} \}_{i=0}^n$ is given by the Fr\'echet (or Karcher) mean,
\[ \varphi(x)=\arg \min_{p\in M} \sum\nolimits_{i=0}^n v_i(x) (\operatorname{dist} (p,m_i)) ^2,\]
where the optimization problem involved can be solved using optimization algorithms developed for matrix manifolds~(see~\citet{AbMaSe2008}, and references therein). The spatial derivatives of the geodesic finite element can be computed in terms of an associated optimization problem. The advantage of the geodesic finite element approach is that it inherits the approximation properties of the underlying linear space finite element, but it can be expensive to compute, since it entails solving an optimization problem on a manifold.

An alternative approach to group-equivariant interpolation for functions taking values on symmetric spaces was introduced in \citet{GaLe2018}, which, in particular, is applicable to the interpolation of Riemannian and pseudo-Riemannian metrics. It uses the generalized polar decomposition~\cite{munthe2001generalized} to construct a local diffeomorphism between a symmetric space and a Lie triple system, and thereby lift a scalar-valued interpolant to a symmetric space-valued interpolant.

\paragraph{\bf Lagrangian Variational Integrators}

Variational integrators~(see \cite{MaWe2001}, and references therein) are a class of geometric structure-preserving numerical integrators that are based on a discretization of Hamilton's principle. They are particularly appropriate for the simulation of Lagrangian and Hamiltonian ODEs and PDEs, as they automatically preserve many geometric invariants, including the symplectic structure, momentum maps associated with symmetries of the system, and exhibit bounded energy errors for exponentially long times.

In the case of Lagrangian ODEs, variational integrators are based on constructing computable approximations $L_d:Q\times Q\rightarrow \mathbb{R}$ of the exact discrete Lagrangian,
\[ L_d^E(q_0,q_1,h)=\ext_{\substack{q\in C^2([0,h],Q) \\ q(0)=q_0, q(h)=q_1}} \int_0^h L(q(t), \dot q(t)) dt,\]
which can be viewed as Jacobi's solution of the Hamilton--Jacobi equation. Given a discrete Lagrangian $L_d$, one introduces the discrete action sum $\mathbb{S}_d=\sum_{k=0}^{n-1} L_d(q_k, q_{k+1})$, and then the discrete Hamilton's principle states that $\delta \mathbb{S}_d=0$, for fixed boundary conditions $q_0$ and $q_n$. This leads to the discrete Euler--Lagrange equations, 
\[D_2 L_d(q_{k-1},q_k)+D_1 L_d(q_k,q_{k+1})=0,
\]
where $D_i$ denotes the partial derivative with respect to the $i$-th argument. This implicitly defines the discrete Lagrangian map $F_{L_d}:(q_{k-1},q_k)\mapsto(q_k,q_{k+1})$ for initial conditions $(q_{k-1},q_k)$ that are sufficiently close to the diagonal of $Q\times Q$.
It is also equivalent to the implicit discrete Euler--Lagrange equations,
\[
p_k=-D_1 L_d(q_k, q_{k+1}),\qquad p_{k+1}=D_2 L_d(q_k, q_{k+1}),\]
which implicitly defines the discrete Hamiltonian map $\tilde{F}_{L_d}:(q_k,p_k)\mapsto(q_{k+1},p_{k+1})$, which is automatically symplectic. This clearly follows from the fact that these equations are precisely the characterization of a symplectic map in terms of a Type I generating function. The two equations in the implicit discrete Euler--Lagrange equations can be used to define the discrete Legendre transforms, \(\mathbb{F}^{\pm}L_{d}: Q\times Q \rightarrow T^{*}Q\):
\begin{align*}
\mathbb{F}^{+}L_{d}&:(q_{0},q_{1}) \rightarrow (q_{1},p_{1}) = (q_{1},D_{2}L_{d}(q_{0},q_{1})), \\
\mathbb{F}^{-}L_{d}&:(q_{0},q_{1}) \rightarrow (q_{0},p_{0}) = (q_{0},-D_{1}L_{d}(q_{0},q_{1})).
\end{align*}
The following commutative diagram illustrates the relationship between the discrete Hamiltonian flow map, discrete Lagrangian flow map, and the discrete Legendre transforms,%
\begin{align*}
\xymatrix{
& (q_{k},p_{k})  \ar[rr]^{\tilde{F}_{L_{d}}} & & (q_{k+1},p_{k+1}) & \\
& & & & \\
(q_{k-1},q_{k}) \ar[uur]^{\mathbb{F}^{+}L_{d}} \ar[rr]_{F_{L_{d}}} & & (q_{k},q_{k+1}) \ar[rr]_{F_{L_{d}}} \ar[uur]^{\mathbb{F}^{+}L_{d}} \ar[uul]_{\mathbb{F}^{-}L_{d}}& &(q_{k+1},q_{k+2}) \ar[uul]_{\mathbb{F}^{-}L_{d}}
}
\end{align*}
If the discrete Lagrangian is invariant under the diagonal action of a Lie group $G$, i.e., $L_d(q_0, q_1)=L_d(gq_0, gq_1)$, for all $g\in G$, then the discrete Noether's theorem states that there is a discrete momentum map that is automatically preserved by the variational integrator. The bounded energy error of variational integrators can be understood by performing backward error analysis~\cite{Ha1994,BeGi1994}, which then shows that the discrete flow map is approximated with exponential accuracy by the exact flow map of the Hamiltonian vector field of a modified Hamiltonian. Similarly, backward error analysis for Lagrangian variational integrators is considered in \cite{Ve2017}.

\paragraph{\bf Multisymplectic Hamiltonian Variational Integrators.}
For Hamiltonian PDEs (see, for example, \citet{MaSh1999}) the action is a functional on the field and multimomenta values (more precisely, sections of the restricted dual jet bundle),
$$ S[\phi,p] = \int [p^\mu \partial_\mu \phi - H(\phi,p) ]d^{n+1}x, $$
where the integration is over some $(n+1)$-dimensional region of spacetime. The variational principle gives the De Donder--Weyl equations $\partial_\mu p^\mu = -\partial H/\partial \phi$, $\partial_\mu\phi = \partial H/\partial p^\mu$. Defining $z = (\phi, p^0, \dots, p^n)$ and $K^{\mu}$ as the $(n+2)\times (n+2)$ skew-symmetric matrix with value $-1$ in the $(0,\mu+1)$ entry, $1$ in the $(\mu+1,0)$ entry, and $0$ in every other entry (with indexing from $0$ to $n+1$), the De Donder--Weyl equations can be written in the form
$$ K^0\partial_0 z + \dots + K^n\partial_n z = \nabla_z H. $$
This formulation of Hamiltonian PDEs was studied in \citet{Br1997}; in particular, it was shown that such a system admits a multisymplectic conservation law of the form $\partial_\mu \omega^\mu(V,W) = 0$, where the $\omega^\mu$ are two-forms corresponding to $K^\mu$ and the conservation law holds when evaluated on first variations $V, W$. For discretizing such equations, multisymplectic integrators have been developed which admit a discrete analogue of this multisymplectic conservation law (see, for example, \citet{BrRe2006}). Such multisymplectic integrators have traditionally not been approached from a variational perspective. 

However, in \citet{TrLe2022}, we developed a systematic method for constructing variational integrators for multisymplectic Hamiltonian PDEs which automatically admit a discrete multisymplectic conservation law and a discrete Noether's theorem by virtue of the discrete variational principle. The construction is based on a discrete approximation of the boundary Hamiltonian that was introduced in~\citet{VaLiLe2011}, 
$$ H_{\partial U}(\varphi_A,\pi_B) = \ext\Big[ \int_B p^\mu \phi d^nx_\mu - \int_U (p^\mu \partial_\mu\phi - H(\phi,p) ) d^{n+1}x \Big],$$
where $\partial U = A \sqcup B$, boundary conditions are placed on the field value $\phi$ on $A$ and normal momenta value on $B$, and one extremizes over the sections $(\phi,p)$ over $U$ satisfying the specified boundary conditions. The boundary Hamiltonian is a generating functional in the sense that the Type II variational principle generates the normal momenta value along $A$ and the field value along $B$,
$$ \frac{\delta H_{\partial U}}{\delta \varphi_A} = -p^n|_A, \quad \frac{\delta H_{\partial U}}{\delta \pi_B} = \phi|_B. $$
A variational integrator is then constructed by first approximating the boundary Hamiltonian using a finite-dimensional function space and quadrature, and subsequently enforcing the Type II variational principle. For example, with particular choices of function spaces and quadrature, \citet{TrLe2022} recover the class of multisymplectic partitioned Runge--Kutta methods.

In this paper, we take a different approach in several regards. First, we focus on Lagrangian field theories as opposed to Hamiltonian field theories. For Hamiltonian field theories, the momenta are related to the field and its derivative by the Legendre transform; this falls out from the variational principle so one does not need to enforce it beforehand. Thus, in this sense, the momenta and field values can be considered as independent before enforcing the variational principle. On the other hand, for Lagrangian field theories, the Lagrangian depends on both the field value and its first derivative, so one cannot na\"\i vely treat the two as independent; that is, the Lagrangian depends on holonomic sections of the jet bundle. As we will see, this will mean that we need to pay particular attention to the holonomic condition when discretizing via a finite element projection. Furthermore, as opposed to constructing variational integrators from a generating functional (the analogue in the Lagrangian framework would be the boundary Lagrangian, see \citet{VaLiLe2011}), in this paper, we instead investigate directly discretizing the variational principle $\delta S = 0$ utilizing projections into finite-dimensional subspaces. Finally, for simplicity, we do not utilize any quadrature approximations of the various integrals which we encounter; for strong nonlinearities in the Lagrangian, one generally has to utilize quadrature to construct an efficient discretization. However, the theory that we outline is also applicable to the case of quadrature approximation by first applying the quadrature approximation of the action before enforcing the variational principle, so that the resulting discretization is still variational; we will elaborate on this in Remark \ref{Remark on Quadrature}. For this reason, we will assume exact integration in order to keep the exposition simple. 

\paragraph{\bf Main Contributions.} This paper studies the variational finite element discretization of Lagrangian field theories from two perspectives; we begin by investigating directly discretizing the full variational principle over the full spacetime domain, which we refer to as the ``covariant" approach, and subsequently study semi-discretization of the instantaneous variational principle on a globally hyperbolic spacetime, which we refer to as the ``canonical" approach. This paper can be considered a discrete analogue to the program initiated in \citet{GoIsMaMo1998, GoIsMaMo2004}, which lays the foundation for relating the covariant and canonical formulations of Lagrangian field theories through their (multi)symplectic structures and momentum maps. One of the goals of understanding the relation between these two different formulations is to systematically relate the covariant gauge symmetries of a gauge field theory to its initial value constraints. This is seen, for example, in general relativity, where the diffeomorphism gauge invariance gives rise to the Einstein constraint equations over the initial data hypersurface (see, for example, \citet{Go2012}). When one semi-discretizes such gauge field theories, the discrete initial data must satisfy an associated discrete constraint. We aim to make sense of the discrete geometric structures in the covariant and canonical discretization approaches as a foundation for understanding the discretization of gauge field theories.

In Section \ref{Covariant Discretization Section}, we begin by formulating a discrete variational principle in the covariant approach, utilizing the finite element construction to appropriately project the variational principle. We show that a cochain projection from the underlying de Rham complex onto the finite element spaces yields a natural discrete variational principle that is compatible with the holonomic jet structure of a Lagrangian field theory. In Section \ref{Variational Structure Section}, we then show that discretizing by cochain projections leads to a naturality relation between the continuous variational problem and the discrete variational problem; this naturality then implies that discretization and the variational principle commute and also, that discretizing at the level of the configuration bundle or at the level of the jet bundle are equivalent. Subsequently, by decomposing the finite element spaces into boundary and interior components, we define a discrete weak Cartan form in analogy with the continuum weak Cartan form which will, in a sense, encode the discrete variational structure. With particular choices of finite element spaces, this discrete weak Cartan form recovers the notion of the discrete Cartan form introduced by \citet{MaPaSh1998}. However, we note that our notion of a discrete weak Cartan form is more general and furthermore, since our discrete variational problem is naturally related to the continuum variational problem, we are able to explicitly discuss in what sense the discrete weak Cartan form converges to the continuum weak Cartan form.  Using this discrete weak Cartan form, in Sections \ref{Multisymplectic Section} and \ref{Noether's Theorem Section}, we state and prove discrete analogues of the multisymplectic form formula and Noether's theorem. In Section \ref{Variational Complex Section}, we reinterpret and concisely summarize the preceding sections by interpreting the discrete variational structures as elements of a discrete variational complex. In Section \ref{Numerical Example Section}, we provide an example of a multisymplectic integrator for the scalar Poisson equation and prove the convergence of the discrete weak Cartan form to the weak Cartan form.

In Section \ref{Semi-Discretization Section}, we study the semi-discretization of the canonical formulation of a Lagrangian field theory on a globally hyperbolic spacetime. In Section \ref{Semi-Discrete EL Section}, we discretize the instantaneous variational principle utilizing cochain projections onto finite element spaces over a Cauchy surface, which gives rise to a semi-discrete Euler--Lagrange equation. In Section \ref{Semi-discrete Symplectic Structure Section}, we relate this semi-discrete Euler--Lagrange equation to a Hamiltonian flow on a symplectic semi-discrete phase space. We will discuss in what sense the symplectic structure on the semi-discrete phase space arises from a symplectic structure on the continuum phase space. Subsequently, we will investigate the energy-momentum map structure associated to the semi-discrete phase space in Section \ref{Energy-Momentum Section}, and discuss how, under appropriate equivariance conditions on the projection, the energy-momentum map structure on the semi-discrete phase space arises as the pullback of the energy-momentum map structure on the continuum phase space. This lays a foundation for understanding initial value constraints when discretizing field theories with gauge symmetries. Finally, in Section \ref{Tensor Product Discretization Section}, we relate the covariant and canonical discretization approaches in the case of tensor product finite element spaces. 

The underlying theme of this paper is that, when one discretizes the variational principle utilizing compatible discretization techniques, the associated (covariant or canonical) discretization inherits discrete variational structures which can be viewed as pullbacks or projections of the associated continuum variational structures. These discrete variational structures allow one to investigate structure-preservation under discretization of important physical properties, such as momentum conservation, symplecticity, and (gauge) symmetries. 

\section{Covariant Discretization of Lagrangian Field Theories}\label{Covariant Discretization Section}
In this section, we discretize the covariant Euler--Lagrange equations which arise from the variational principle $\delta S[\phi] = 0$ for the action $S: \phi \mapsto \int_X \mathcal{L}(j^1_d\phi)$ where $\phi \in H\Lambda^k$ is an element of the configuration space and $j^1_d\phi = (x,\phi,d\phi)$. To utilize the finite element method, we take our base space $X$ to be a bounded $(n+1)-$dimensional polyhedral domain with boundary $\partial X$, equipped with a finite element triangulation $\mathcal{T}_h$. We will assume $X$ has a Riemannian or Lorentzian metric. For this discretization, we perform the variation over a finite element space, and subsequently study how the multisymplectic and covariant momentum map structures are affected by discretization. In particular, we show how these structures are preserved for particular choices of finite element spaces, namely spaces whose projections are cochain maps or group-equivariant interpolation spaces. To begin, we first discuss the weak formulation of Lagrangian field theory.

\subsection{Weak Lagrangian Field Theory}\label{Weak Lagrangian Field Theory Section}
In this section, we formulate a weak version of Lagrangian field theory on the Hilbert space $H\Lambda^k$. Since we wish to work in the Sobolev space setting, it does not make sense to consider pointwise values of (e.g., square integrable) sections. However, we will assume that the Lagrangian density makes sense as a map on sections, $\mathcal{L}: J^1_{H\Lambda^k} \rightarrow \Omega^{n+1}(X)$, i.e., given a section $\phi \in H\Lambda^k$, the quantity $\mathcal{L}(j^1_d\phi)$ is a top-dimensional form on $X$. Hence, we can define the action $S: H\Lambda^k \rightarrow \mathbb{R}$ via $S[\phi] = \int_X \mathcal{L}(j^1_d\phi)$. Thus, from our perspective, a weak Lagrangian field theory is defined by a Lagrangian density $\mathcal{L}: J^1_{H\Lambda^k} \rightarrow \Omega^{n+1}(X)$ with associated action $S: H\Lambda^k \rightarrow \mathbb{R}$.

We derive the weak Euler--Lagrange equations in the Hilbert space setting, where the velocity phase space is $J^1_{H\Lambda^k} = H\Lambda^k \times dH\Lambda^k.$ Fixing the trace of $\phi$ on $\partial X$, the variational principle is to find $\phi \in H\Lambda^k$ such that $\delta S[\phi]\cdot v = 0$ for all $v \in \mathring{H}\Lambda^k \equiv \{ v \in H\Lambda^k : \text{Tr}(v) = 0\}$. This yields the weak Euler--Lagrange equations
\begin{align}\label{Weak EL Equations}
0 &= \delta S[\phi]\cdot v = \int_X \Big(\delta_2\mathcal{L}(j^1_d\phi)\cdot v + \delta_3\mathcal{L}(j^1_d\phi)\cdot dv\Big)\\
&= (\partial_2\mathcal{L}(j^1_d\phi),v)_{L^2\Lambda^k} + (\partial_3\mathcal{L}(j^1_d\phi),dv)_{L^2\Lambda^{k+1}} \nonumber \\
&= (\partial_2\mathcal{L}(j^1_d\phi),v)_{L^2\Lambda^k} + (d^*\partial_3\mathcal{L}(j^1_d\phi),v)_{L^2\Lambda^k},\nonumber
\end{align}
where $\delta_i$ denotes the variation with respect to the $i^{th}$ argument, the codifferential $d^*$ is interpreted in the weak sense, and in the second line we apply the Riesz representation theorem to express the linear functional $v \in L^2\Lambda^k \mapsto \int_X \delta_2\mathcal{L}(j^1_d\phi )\cdot v$ as an element $\partial_2\mathcal{L}(j^1_d\phi)$ of $L^2\Lambda^k$ and similarly, the linear functional $w \in L^2\Lambda^{k+1} \mapsto \int_X \delta_3\mathcal{L}(j^1_d\phi) \cdot w$ as an element $\partial_3\mathcal{L}(j^1_d\phi)$ of $L^2\Lambda^{k+1}$, assuming that these linear functionals are bounded.

\begin{remark}\label{Lagrangian Regularity Remark}
As mentioned above, the linear functionals $v \in L^2\Lambda^k \mapsto \int_X \delta_2\mathcal{L}(j^1_d\phi )\cdot v$ and $w \in L^2\Lambda^{k+1} \mapsto \int_X \delta_3\mathcal{L}(j^1_d\phi) \cdot w$ should be bounded in order to represent them as elements of $L^2\Lambda^k$ and  $L^2\Lambda^{k+1}$, respectively. We give some examples of classes of Lagrangian densities for which this holds.

Consider a Lagrangian density containing at most quadratic terms in $\phi$ and $d\phi$, of the form
$$ \mathcal{L}(j^1_d\phi) = \frac{1}{2} a_1 d\phi \wedge * d\phi + \frac{1}{2} a_2 \phi \wedge * \phi + a_3 f \wedge * d\phi + a_4 g \wedge * \phi, $$
where $a_i \in L^\infty, f \in L^2\Lambda^{k+1}, g \in L^2\Lambda^k$ are given. The variation of the associated action can be computed
\begin{align*}
\delta S[\phi] \cdot v = \int_X \Big((a_2\phi + a_4 g)\wedge * v + (a_1 d\phi + a_3 f) \wedge * dv\Big).
\end{align*}
We see that the functional $v \in L^2\Lambda^k \mapsto \int_X \delta_2\mathcal{L}(j^1_d\phi )\cdot v = \int_X (a_2\phi + a_4 g)\wedge * v$ is bounded, since
$$ \Big|\int_X (a_2\phi + a_4 g)\wedge * v\Big| = (a_2\phi, v)_{L^2\Lambda^k} + (a_4 g, v)_{L^2\Lambda^{k}} \leq (\|a_2\|_{L^\infty}\|\phi\|_{L^2\Lambda^k} + \|a_4\|_{L^\infty}\|g\|_{L^2\Lambda^k}) \|v\|_{L^2\Lambda^k}. $$
Thus, we can represent this functional as an element of $L^2\Lambda^k$; explicitly, $\partial_2\mathcal{L}(j^1_d\phi) = a_2\phi + a_4g$. Similarly, $w \in L^2\Lambda^{k+1} \mapsto \int_X \delta_3\mathcal{L}(j^1_d\phi) \cdot w = \int_X (a_1 d\phi + a_3 f) \wedge * w$ is bounded and $\partial_3\mathcal{L}(j^1_d\phi) = a_1 d\phi + a_3 f$.

One can also consider nonlinearities, given sufficient control on the nonlinearity. For example, with $k=0$ (for simplicity; one could consider $k \geq 1$ with the nonlinearities acting on the components of $\phi$), consider a Lagrangian density which contains a term of the form $V(\phi)\, d^{n+1}x$, where $V \in C^1(\mathbb{R}, \mathbb{R})$ has bounded derivative $V' \in L^\infty(\mathbb{R},\mathbb{R})$. The variation of this term in the associated action gives the linear functional
$$ v \in L^2 \mapsto \int_X V'(\phi) v\, d^{n+1}x. $$
Since the domain $X$ is bounded, we have the continuous embedding $L^2 \hookrightarrow L^1$ with $\| v \|_{L^1} \leq C \| v\|_{L^2}$. Hence, the above linear functional is bounded, since
$$ \Big| \int_X V'(\phi) v\, d^{n+1}x \Big| \leq \|V'\|_{L^\infty} \|v\|_{L^1} \leq C\|V'\|_{L^\infty} \|v\|_{L^2}. $$
An example of such a nonlinearity occurs in the sine--Gordon Lagrangian density, which contains a term of the form $V(\phi)\, d^{n+1}x = \cos(\phi)\, d^{n+1}x.$
\end{remark}

We now define a weak analogue of the Cartan form \eqref{Cartan form}, relative to a region $U \subset X$. If we only assume $H\Lambda$ regularity on the fields and variations, we define the weak Cartan form, at a solution of the weak Euler--Lagrange equations, to be the variation of the action
$$ \Theta_U(\phi)\cdot v \equiv dS[\phi]\cdot v; $$
note that this is in general nonzero since we are not assuming that $v$ has vanishing trace on the boundary. In some sense, the weak Cartan form encodes the contribution of the boundary term to the variation of the action. To see this explicitly, we need to assume higher regularity.

To make sense of such a boundary term, we require higher regularity, at least locally on $U$; namely, since the trace acts as a bounded operator $\text{Tr}: H\Lambda^m(U) \rightarrow H^{-1/2}\Lambda^m(\partial U)$ (\citep{ArFaWi2006}) and as a bounded operator $\text{Tr}: H^1\Lambda^m(U) \rightarrow H^{1/2}\Lambda^m(\partial U)$, the solution $\phi$ and the Lagrangian have to have enough regularity so that $\partial_3\mathcal{L}(j^1_d\phi)$ is in $H^1\Lambda^{k+1}(U)$. For example, in the first class of Lagrangians discussed in Remark \ref{Lagrangian Regularity Remark}, if the solution $\phi$ has $H^2$ regularity and the given $f$ has $H^1$ regularity, then this is satisfied. Assuming this higher regularity, the weak Cartan form is defined to be the boundary term which arises for a variation $v$ with generally nonzero boundary trace. That is,
\begin{equation}\label{Weak Cartan Form}
\Theta_U(\phi)\cdot v \equiv \int_{\partial U} v \wedge * \partial_3\mathcal{L}(j^1_d\phi).
\end{equation}

We refer to this as the weak Cartan form since it involves integration, whereas (in the smooth setting) the Cartan form is the integrand of the above expression. With this definition, the variation of the action with respect to $v \in H\Lambda^k$ can be expressed
$$ \delta S[\phi] \cdot v = \text{EL}(\phi) \cdot v + \Theta(\phi) \cdot v, $$
where $\text{EL}(\phi) \cdot v \equiv (\partial_2\mathcal{L}(j^1_d\phi),v)_{L^2\Lambda^k} + (d^*\partial_3\mathcal{L}(j^1_d\phi),v)_{L^2\Lambda^k}$ is the weak Euler--Lagrange form which, by definition, vanishes for a solution $\phi$ of the weak Euler--Lagrange equations.

It will also be useful to think of variations as vector fields over the configuration space. With the identification $T(H\Lambda^k) \cong H\Lambda^k \times H\Lambda^k$, we can view a vector field $V \in \mathfrak{X}(H\Lambda^k)$ as a map $V: H\Lambda^k \rightarrow H\Lambda^k$. Thus, we define the weak Cartan form and weak Euler--Lagrange form, acting on vector fields, as
\begin{align*}
\mathbf{\Theta}(\phi) \cdot V \equiv \Theta(\phi) \cdot V(\phi), \\
\textbf{EL}(\phi) \cdot V \equiv \text{EL}(\phi) \cdot V(\phi).
\end{align*}
The variation of the action with respect to $V$ can then be expressed $dS[\phi]\cdot V = \textbf{EL}(\phi) \cdot V + \mathbf{\Theta}(\phi) \cdot V$. With the above notation, we now derive weak analogues of the multisymplectic form formula and Noether's theorem.

\textbf{Weak Multisymplectic Form Formula.}
Let $V,W$ be first variations of a solution $\phi$ of the weak Euler--Lagrange equations, i.e., their respective flows on $\phi$ still satisfy the weak Euler--Lagrange equation. Then, one has the weak multisymplectic form formula
\begin{equation}\label{Weak Multisymplectic Form Formula}
d\Theta(\phi) \cdot (V,W) = 0.
\end{equation}
The proof follows from $d^2S(\phi)\cdot (V,W) = 0$. We will perform the proof in the discrete setting in Theorem \ref{Discrete Multisymplectic Form Theorem}, where the computation is analogous.

\textbf{Weak Noether's Theorem.} Suppose there is a Lie group action of a Lie group $G$ on $H\Lambda^k$, which we denote by $g \cdot \phi$ for $g \in G, \phi \in H\Lambda^k$. For a Lie algebra element $\xi \in \text{Lie}(G)$, we denote by $\widetilde{\xi}$ its associated infinitesimal generator, which is a vector field on $H\Lambda^k$ defined by 
$$ \widetilde{\xi}(\phi) = \lim_{t\rightarrow 0} \frac{e^{t\xi} \cdot \phi - \phi}{t}. $$
Furthermore, suppose that the action $S_U: H\Lambda^k \rightarrow \mathbb{R}$ is $G$-invariant for any region $U \subset X$, i.e., $S_U[g\cdot \phi] = S[\phi]$ for all $g \in G, \phi \in H\Lambda^k$. Thus, $S_U[e^{t\xi} \cdot \phi] = S_U[\phi]$ for all $\xi \in \text{Lie}(G)$. By differentiating, this gives the expression
$$ dS_U[\phi]\cdot \widetilde{\xi} = 0 \text{ for all } \xi \in \text{Lie}(G). $$
Explicitly, one has
\begin{align*}
0 &= dS_U[\phi]\cdot \widetilde{\xi} = (\partial_2\mathcal{L}(j^1_d\phi), \widetilde{\xi}(\phi))_{L^2\Lambda^k(U)} + (\partial_3\mathcal{L}(j^1_d\phi), d\widetilde{\xi}(\phi))_{L^2\Lambda^{k+1}(U)} \\
&= (\partial_2\mathcal{L}(j^1_d\phi), \widetilde{\xi}(\phi))_{L^2\Lambda^k} + (d^*\partial_3\mathcal{L}(j^1_d\phi), \widetilde{\xi}(\phi))_{L^2\Lambda^{k+1}(U)} + \int_{\partial U} \widetilde{\xi}(\phi) \wedge \star \partial_3\mathcal{L}(j^1_d\phi).
\end{align*}
The first two terms in the second line above vanish by the weak Euler--Lagrange equation, so that
$$ 0 = \int_{\partial U} \widetilde{\xi}(\phi) \wedge \star \partial_3\mathcal{L}(j^1_d\phi) = \Theta_U(\phi)\cdot \widetilde{\xi}.  $$
Thus, Noether's theorem in the weak setting states that the integrated Cartan form paired with an infinitesimal generator of a $G$ action vanishes, $\Theta_U(\phi)\cdot \widetilde{\xi} = 0$, if the action is $G$-invariant. In the smooth setting, by applying Stoke's theorem and noting that $U$ is arbitrary, one has the stronger statement that the exterior derivative of the integrand above vanishes (\citet{MaPaSh1998}).

\subsection{Variational Discretization}\label{Variational Structure Section}

To formulate a discrete variational principle, let $\{\Lambda^m_h\}_{m=0}^{n+1}$ be a subcomplex of finite element spaces approximating $\{H\Lambda\}$ with projections $\pi^m_h: H\Lambda^m \rightarrow \Lambda^m_h$. This provides an approximation of $J^1_{H\Lambda^k} = H\Lambda^k \times dH\Lambda^k$ by $\pi^k_hH\Lambda^k \times \pi^{k+1}_h(dH\Lambda^k)$. Consider the degenerate Lagrangian density $\mathcal{L}_h: J^1_{H\Lambda^k} \rightarrow \Omega^{n+1}(X)$, $\mathcal{L}_h(j^1_d\phi) \equiv \mathcal{L}(x,\pi^k_h\phi, \pi^{k+1}_hd\phi)$ and the associated degenerate action $S_h: H\Lambda^k \rightarrow \mathbb{R}$ defined by
\begin{equation}\label{Degenerate Action, Derivative First}
S_h[\phi] = \int_X \mathcal{L}(x,\pi^k_h\phi,\pi^{k+1}_hd\phi).
\end{equation}
We refer to these as degenerate since the projections have nontrivial kernels, as projections from infinite-dimensional spaces to finite-dimensional subspaces.

The variational principle associated to the degenerate action $S_h$ is to find $\phi \in H\Lambda^k$ such that
\begin{equation}\label{Degenerate Action DEL, Derivative First}
0 = \delta S_h[\phi] \cdot v = (\partial_2\mathcal{L}_h(j^1_d\phi), \pi^k_h v)_{L^2\Lambda^k} + (\partial_3\mathcal{L}_h(j^1_d\phi), \pi^{k+1}_h dv)_{L^2\Lambda^{k+1}}, \text{ for all } v \in \mathring{H}\Lambda^k.
\end{equation}
The issue with \eqref{Degenerate Action DEL, Derivative First} is that the in the second term on the right hand side, the projection $\pi^{k+1}_h dv$ occurs after taking the exterior derivative, so one cannot in general integrate by parts to obtain a boundary term, which is necessary in the continuous theory to define the Cartan form (which, recall, is defined to be the boundary term induced by a variation which does not vanish on the boundary). 

On the other hand, one can produce the desired boundary term if one instead utilizes a different degenerate action defined by $\widetilde{S}_h \equiv S \circ \pi^k_h$,
\begin{equation}\label{Degenerate Action, Project First}
\widetilde{S}_h [\phi] = \int_X \mathcal{L}(x,\pi^k_h\phi, d\pi^k_h\phi),
\end{equation}
since the associated variational principle is to find $\phi \in H\Lambda^k$ such that 
\begin{equation}\label{Degenerate Action DEL, Project First}
0 = \delta \widetilde{S}_h [\phi] = (\partial_2\mathcal{L}(x,\pi^k_h\phi, d\pi^k_h\phi), \pi^k_hv)_{L^2\Lambda^k} + (\partial_3\mathcal{L}(x,\pi^k_h\phi, d\pi^k_h\phi), d\pi^k_hv)_{L^2\Lambda^k}, \text{ for all } v \in \mathring{H}\Lambda^k.
\end{equation}
One can now integrate by parts in the second term, since the exterior derivative is taken after the projection. However, the issue with the latter degenerate action, $\widetilde{S}_h$, is that there is in general no associated degenerate Lagrangian density, i.e., there is in general no map $\widetilde{\mathcal{L}}_h: J^1_{H\Lambda^k} \rightarrow \Omega^{n+1}(X)$ such that $\widetilde{\mathcal{L}}_h(j^1_d\phi) = \mathcal{L}(x,\pi^k_h\phi, d\pi^k_h\phi)$. One would want there to be an associated degenerate Lagrangian density, in order to compare to the continuous theory, e.g., when examining convergence. 

Thus, the degenerate action $S_h$ has the issue that one cannot in general extract a boundary term in the variation, whereas the degenerate action $\widetilde{S}_h$ has the issue that one cannot in general associate to it a degenerate Lagrangian density. Both of these issues are resolved with the assumption that the projections commute with the exterior derivative, $\pi^{k+1}_hd\phi = d\pi^k_h\phi$, since then $\widetilde{S}_h = S_h$. We will henceforth assume this through the paper.

\begin{assumption}[Cochain Projections]\label{Cochain Projection Assumption}
The projections $\pi^m_h: H\Lambda^m \rightarrow \Lambda^m_h$ are cochain projections, i.e., $\pi^{k+1}_hd = d\pi^k_h$.
\end{assumption}
Furthermore, we will generally denote the projections as $\pi_h$, where the degree of the differential forms that they act on are suppressed for notational convenience. 

With this assumption, the two variational principles \eqref{Degenerate Action DEL, Derivative First} and \eqref{Degenerate Action DEL, Project First} are equivalent. However, even ignoring issues of degeneracy of the Lagrangian density itself, e.g., due to gauge freedom, these equivalent variational principles are underdetermined due to the nontrivial kernels of the projections. As such, the action is constant on fibers of the projection, which corresponds to a symmetry of the action. Thus, instead of enforcing the variational principle over the full field space, the finite-dimensional reduction to the problem is given by enforcing the variational principle over the discrete space: find $\phi \in \Lambda^k_h$ such that $\delta S[\phi]\cdot v = 0$ for all $v \in \Lambda^k_h$ with vanishing trace on the boundary; we denote the space of such $v$ by $\mathring{\Lambda}^k_h$. The variational principle thus yields a discrete weak form of the Euler--Lagrange equation: find $\phi \in \Lambda^k_h$ such that
\begin{equation}\label{DEL 1a}
0 = \delta S[\phi]\cdot v = \big(\partial_2\mathcal{L}(j^1_d\phi),v\big)_{L^2\Lambda^k} + \big(\partial_3\mathcal{L}(j^1_d\phi),dv\big)_{L^2\Lambda^{k+1}}, \text{ for all } v \in \mathring{\Lambda}^k_h.
\end{equation}
Integrating by parts, this gives
\begin{equation}\label{DEL 1b}
0 = \big(\partial_2\mathcal{L}(j^1_d\phi),v\big)_{L^2\Lambda^k} + \big(d^*\partial_3\mathcal{L}(j^1_d\phi),v\big)_{L^2\Lambda^{k}} + \int_{\partial X}v \wedge * \partial_3\mathcal{L}(j^1_d\phi), \text{ for all } v \in \mathring{\Lambda}^k_h,
\end{equation}
where the codifferential $d^*$ is interpreted in the weak sense. Note the boundary term vanishes since $v \in \mathring{\Lambda}^k_h$, but we include it explicitly since it will be necessary in the formulation of the multisymplectic form formula and Noether's theorem, where one generally has nonzero variations on the boundary.

We refer to these equivalent equations, (\ref{DEL 1a}) and (\ref{DEL 1b}), as the discrete Euler--Lagrange equations (DEL). Fixing a basis of shape functions $\{v_i\}$ for $\mathring{\Lambda}^k_h$, expressing $\phi = \phi^j v_j$, and choosing $v = v_i$, (\ref{DEL 1a}) is equivalent to a (generally nonlinear) system of equations for the unknown components $\phi^i$. Letting $[i]$ denote the set of indices $j$ such that $\text{supp}(v_j) \cap \text{supp}(v_i)$ has positive measure, the system of equations can be written as
$$ \big(\partial_2\mathcal{L}(j^1_d( \sum_{j \in [i]} \phi^j v_j)),v_i\big)_{L^2\Lambda^k} + \big(\partial_3\mathcal{L}(j^1_d(\sum_{j \in [i]}  \phi^j v_j)),dv_i\big)_{L^2\Lambda^{k+1}} = 0,\ i=1,\dots,\dim \mathring{\Lambda}^k_h. $$

In order to provide local statements of the multisymplectic form formula and Noether's theorem, we now localize the DEL. For a region $U \subset X$, we say that a node $i$ is an interior point of $U$ if $U$ contains all simplices touching $i$. Denote $\bar{U}$ as the union of all simplices touching interior nodes $i$ of $U$; we say that $U$ is regular if $U = \bar{U}$. We define the admissible variations with respect to a regular region $U$ as the space of all $v \in \mathring{\Lambda}^k_h$ such that $v|_U \in \mathring{\Lambda}^k_h(U)$. We define the localized action $S_U[\phi] = \int_U \mathcal{L}(j^1_d\phi)$ and the associated localized DEL,
\begin{align}\label{local DEL} 0 = \delta S_U[\phi]\cdot v &= \big(\partial_2\mathcal{L}(j^1_d\phi),v\big)_{L^2\Lambda^k(U)} + \big(\partial_3\mathcal{L}(j^1_d\phi),dv\big)_{L^2\Lambda^{k+1}(U)}  \\
&= \big(\partial_2\mathcal{L}(j^1_d\phi),v\big)_{L^2\Lambda^k(U)} + \big(d^*\partial_3\mathcal{L}(j^1_d\phi),v\big)_{L^2\Lambda^{k}(U)} + \int_{\partial U}v \wedge * \partial_3\mathcal{L}(j^1_d\phi), \nonumber
\end{align}
which is enforced for all regular $U$ and admissible $v$. As before, the boundary term vanishes for admissible $v$, but we write it explicitly as it will arise later. 
\begin{prop}
The localized DEL (\ref{local DEL}), ranging over all regular $U$ and admissible $v$, are equivalent to the DEL (\ref{DEL 1b}).
\begin{proof}
To see that the localized DEL imply the DEL, choose $U = X$ which is trivially regular; the space of admissible variations with respect to $X$ is then just $\mathring{\Lambda}^k_h$. To see that the DEL imply the localized DEL, let $U$ be regular and $v$ be admissible. Since $\text{supp}(v) \subset U$, the integrals over $X$ in the DEL can be replaced by integrals over $U$.
\end{proof}
\end{prop}

In this section, we aim to elucidate the variational structure that arises from discretizing the variational principle utilizing cochain projections. Recalling that the Cartan form (\ref{Cartan form}) encodes the variational structure of a Lagrangian field theory, we will construct a discrete analogue of the Cartan form, which will naturally encode the variational structure of the discretized theory. 

We first show that the restricted variational principle over the finite-dimensional subspace $\mathring{\Lambda}^k_h$ can be interpreted as a Galerkin variational integrator. Restricting the configuration space to $\mathring{\Lambda}^k_h$, we can view the action as a function of the components $\phi^i$ in the expansion $\phi = \phi^i v_i$.
$$ S[\phi^i] = \int \mathcal{L}(x,\phi^i v_i, \phi^i dv_i). $$
Taking the variation of $S$ with respect to $\phi^j$, 
\begin{align*}
\frac{\delta S[\phi^i]}{\delta \phi^j} &= \int \Big( \frac{\delta \mathcal{L}}{\delta \phi} \cdot \frac{\delta (\phi^iv_i)}{\delta \phi^j} + \frac{\delta \mathcal{L}}{\delta (d\phi)} \cdot \frac{\delta (\phi^idv_i)}{\delta \phi^j} \Big) = \int \Big( \frac{\delta \mathcal{L}}{\delta \phi} \cdot v_j + \frac{\delta \mathcal{L}}{\delta (d\phi)} \cdot dv_j \Big) \\
&= (\partial_2\mathcal{L},v_j) + (\partial_3\mathcal{L},dv_j),
\end{align*}
which shows that the conditions $\delta S/\delta \phi^j = 0$ is equivalent to the DEL (\ref{DEL 1b}). Similarly, the localized DEL (\ref{local DEL}) is equivalent to the conditions $\delta S_U/\delta \phi^j = 0$ for all interior nodes $j$. That is, the DEL can be interpreted as a Galerkin variational integrator. From this viewpoint of the DEL, we see that given appropriate choices of function spaces (and possibly a choice of quadrature rule), our discrete Euler--Lagrange equation reproduces multisymplectic variational integrators based on finite differences or nodal value finite element spaces (e.g., as discussed in \citet{MaPaSh1998} and \citet{Ch2008}). However, the discrete variational principle in the form $\delta S[\phi]\cdot v = 0$, for $\phi \in \Lambda^k_h$ and $v \in \mathring{\Lambda}^k_h$, is expressed explicitly at the level of function spaces and hence, will allow us to examine the discrete variational structure more directly. Along with allowing more general approximating finite element spaces, this also has the advantage of stating properties of the discrete variational principle at the level of function spaces. Consequently, as we will see, properties such as multisymplecticity and Noether's theorem can be stated in a geometric way, which makes no explicit reference to finite differencing or quadrature.

By the above, we can view the Lagrangian structure associated to the equations (\ref{DEL 1a}) as the restriction of the full Lagrangian structure to the discrete space. The next natural question to ask would be: is there some sense in which the discrete equations, which arises as a restriction of the variational principle, can instead be viewed as a variational principle on the full configuration bundle? Since we assume that the projection maps $\pi_h: H\Lambda^m \rightarrow \Lambda^m_h$ are cochain projections on the Hilbert de Rham complex, there is a natural relation between the dynamics of the restricted Lagrangian structure and variations on the full space of the degenerate Lagrangian. To see this, recall that we view the Lagrangian density as a map on the space of sections, $\mathcal{L}: J^1_{H\Lambda^k} \rightarrow \Omega^{n+1}(X)$. Furthermore, recall the degenerate Lagrangian density, $\mathcal{L}_h: J^1_{H\Lambda^k} \rightarrow \Omega^{n+1}(X)$ given by $\mathcal{L}_h(j^1_d\phi) = \mathcal{L}(x,\pi_h\phi,\pi_h d\phi)$ with associated degenerate action $S_h[\phi] = \int_X \mathcal{L}_h(j^1_d\phi)$. In the case of a cochain projection, we can then view the variations of $S$ restricted to $\Lambda^k_h$ as variations of $S_h$ on the full configuration bundle. 

\begin{prop}{\textbf{(Naturality of Discrete Variational Structure)}}\label{Naturality}
\\ The restricted variational structures are related to the degenerate variational structures by
\begin{align}
\mathcal{L}(j^1_d\pi_h\phi ) &= \mathcal{L}_h(j^1_d\phi), \label{Naturality Eqn 1} \\ 
\delta S[\pi_h\phi]\cdot \pi_hv &= \delta S_h[\phi]\cdot v, \label{Naturality Eqn 2}
\end{align}
for $\phi \in H\Lambda^k$ and $v \in H\Lambda^k$.
\begin{proof}
For (\ref{Naturality Eqn 1}), since $\pi_h$ is a cochain projection,
$$ \mathcal{L}(j^1_d\pi_h\phi) = \mathcal{L}(x,\pi_h\phi, d\pi_h\phi) = \mathcal{L}(x,\pi_h\phi, \pi_h d\phi) = \mathcal{L}_h(x,\phi,d\phi) = \mathcal{L}_h(j^1_d\phi). $$
Thus, it follows that $S[\pi_h\phi] = S_h[\phi]$.

Then, (\ref{Naturality Eqn 2}) follows similarly using the cochain property,
\begin{align*}
\delta S[\pi_h\phi]\cdot \pi_h v &= \frac{d}{d\epsilon}\Big|_{\epsilon = 0} S[\pi_h\phi + \epsilon \pi_hv)] = \frac{d}{d\epsilon}\Big|_{\epsilon = 0} S[\pi_h(\phi + \epsilon v)]
\\ &= \frac{d}{d\epsilon}\Big|_{\epsilon = 0} S_h[\phi + \epsilon v] = \delta S_h[\phi]\cdot v.
\end{align*}
\end{proof}
\end{prop}

\begin{remark}
The above proposition is subtle, in that there are two degenerate actions that one could define, recalling $S_h$ and $\widetilde{S}_h$ defined by \eqref{Degenerate Action DEL, Derivative First} and \eqref{Degenerate Action DEL, Project First}, respectively. 
which are both maps from $H\Lambda^k \rightarrow \mathbb{R}$. As discussed previously, the latter is not, without the cochain projection assumption, holonomic in the sense that it does not implicitly depend on $\phi$ through its exterior jet extension, $j^1_d\phi = (x,\phi,d\phi)$, while the former is. Thus, the former degenerate action is the more suitable degenerate action when comparing the discrete theory to the continuous theory, due to its holonomic dependence on $j^1_d\phi$. On other hand, as previously remarked, the former action has the issue that, without the cochain property, one cannot make sense of a boundary term in the variation (which we will need to make sense of a discrete Cartan form), whereas one can in the latter. Assuming a cochain projection, these respective issues of the two degenerate actions are both solved simultaneously, since $S_h = S \circ \pi_h$. 
\end{remark}

The naturality equations \eqref{Naturality Eqn 1} and \eqref{Naturality Eqn 2} reveal that the process of discretization of the variational principle, i.e., by restricting the action and its variations to a finite-dimensional subspace, with the assumption of cochain projections for discretization, is itself associated to an action which arises from a holonomic Lagrangian density on the full field space. Simply put, the discretization is compatible with the structure of a Lagrangian theory. A corollary is that equation \eqref{DEL 1a} can be seen as either arising from the discrete variation of the full action $S$ at a discrete field $\pi_h\phi$, or as from the full variation of the discrete action $S_h$ at the full field $\phi$. This shows that the variations associated to $S_h$ on the full field space are degenerate, since they are equivalently given by the variations of $S$ on the projected space. Thus, the finite-dimensionality of the restricted variational principle on $S$ can be interpreted as the degeneracy of the variational principle of $S_h$ on the full space, where two fields are equivalent if their difference is in $\ker(\pi_h)$. In other words, our finite-dimensional variational problem on the discrete space arises as a degenerate (symmetric) variational problem over the infinite-dimensional space, where the set of equivalence classes forms a finite-dimensional space, with the canonical representative $i_h\pi_h\phi$ for the equivalence class of $\phi$, where $i_h:\Lambda^k_h \hookrightarrow H\Lambda^k$ is the inclusion map. 

Furthermore, the above naturality relation shows that projecting the equations obtained from the variational principle applied to the continuum action is equivalent to first discretizing the action through the projection and subsequently applying the variational principle. Thus, when discretizing via cochain projections, the variational principle and discretization commute:
\[\begin{tikzcd}
	{S: H\Lambda^k \rightarrow \mathbb{R}} &&& {S_h: \Lambda^k_h \rightarrow \mathbb{R}} \\
	\\
	\\
	{\text{Weak EL}} &&& {\text{Discrete EL .}}
	\arrow["{\substack{\text{Variational} \\ \text{Principle}}}"', from=1-1, to=4-1]
	\arrow["{\text{Discretize}}", from=1-1, to=1-4]
	\arrow["{\quad \text{Discretize}}", from=4-1, to=4-4]
	\arrow["{\substack{\text{Variational} \\ \text{Principle}}}", from=1-4, to=4-4]
\end{tikzcd}\]
This generalizes the result of \citet{Le2004} where it was shown that discretization via discrete exterior calculus and the variational principle commute in the case of electromagnetism. In particular, the result of \citet{Le2004} follows from the above, since one can view discrete exterior calculus in the framework of finite element exterior calculus as a particular low-order example; namely, through the use of Whitney forms. 

As a final remark on the above naturality relation, a more fundamental issue for discretization is whether one should discretize at the level of the configuration bundle or the jet bundle. One can discretize the field first via $\phi \mapsto \pi_h\phi$ and take the argument of the Lagrangian density to be $j^1_d\pi_h\phi = (x,\pi_h\phi,d\pi_h\phi)$, or one can take the argument of the Lagrangian density to be $(x, \pi_h\phi, \pi_hd\phi)$; in general, these methods are not equivalent. However, in the case of cochain projections, these two discretization processes are equivalent, i.e., the following diagram commutes

\centerline{ \xymatrix@C+5pc{ \phi \ar@{|->}[r]^{\pi_h^k} \ar@{|->}[d]^{j^1_d} & \pi_h\phi \ar@{|->}[d]^{j^1_d} \\ j^1\phi \ar@{|->}[r]^{\pi^k_h \times \pi^{k+1}_h} & j^1_d(\pi_h\phi),  } }

so there is no ambiguity as to which discretization procedure to use. Furthermore, regarding Assumption \ref{Cochain Projection Assumption}, the above diagram shows that we only need the existence of the space $\Lambda^{k+1}_h$ and the projection $\pi^{k+1}_h$ such that the above diagram commutes and thus, one can perform the discretization solely using $\Lambda^k_h$ and $\pi^k_h$, without reference or implementation of $\Lambda^{k+1}_h$ and $\pi^{k+1}_h$. In particular, as discussed in, for example, \citet{ArFaWi2006, ArFaWi2010} and \citet{Ar2018}, there is a large class of classical finite element spaces for which such cochain projections exist, so the discussion is broadly applicable.

In order to state discrete analogues of the multisymplectic form formula and Noether's theorem, we will have to consider variations with nonzero boundary trace with respect to a regular region $U$. To do this, let $U$ be a regular region and let $v \in \Lambda^k_h$, and consider $v$ restricted to $U$. In general, since we no longer assume that $v$ is an admissible variation relative to $U$, $v$ may have nonzero trace along $\partial U$. Decompose $v = v_\partial + v_{in}$ where $v_\partial$ denotes the boundary component of $v$ consisting of the expansion of $v$ with respect to all shape functions which have nonzero trace on $\partial U$ and $v_{in} = v - v_\partial$ corresponds to the expansion of $v$ into shape functions with vanishing trace on the boundary. Let $\mathcal{T}[\partial U]$ denote the set of all top-dimensional elements in $\mathcal{T}_h$ on which shape functions with nonvanishing trace on $\partial U$ are supported. 

\begin{remark}
If one considers Lagrange polynomial nodal shape functions (corresponding to point value degrees of freedom), then the shape functions which are nonzero on the boundary are those associated to the nodes on $\partial U$. In this case, $\mathcal{T}[\partial U]$ consists of those top-dimensional elements touching the boundary, i.e., the one-ring of the boundary $\partial U$. For general (local) shape functions, internal nodes may give rise to shape functions which are nonzero on the boundary, so $\mathcal{T}[\partial U]$ will generally consist of the elements touching $\partial U$ and the elements touching those elements, i.e., the two-ring of the boundary $\partial U$. In any case, we consider discretization by the finite element method due to the local support property of the shape functions, which will allow the discrete Cartan form defined below to be localized on $\mathcal{T}[\partial U]$.
\end{remark}

We can now consider variations with nonvanishing trace on $\partial U$. In particular, we compute for a solution $\phi_h$ of the discrete Euler--Lagrange equation and for a variation $v$,
\begin{align*}
\delta S_U[\phi_h]\cdot v &=  \sum_{T \in \mathcal{T}[\partial U]} \int_T (\partial_2\mathcal{L}(j^1\phi_h)\wedge *v_\partial + \partial_3\mathcal{L}(j^1\phi_h)\wedge\star d v_{\partial}) = \delta S_U[\phi_h]\cdot v_{\partial},
\end{align*}
i.e., for a solution of the DEL, $\delta S_U[\phi_h]\cdot v = \delta S_U[\phi_h]\cdot v_{\partial}$, since $\delta S_U[\phi_h]\cdot v_{in} = 0$ by the DEL. This boundary variation formula will be our candidate for a discrete weak Cartan form, as it encodes the contribution to the action from $V$ nonvanishing on and near the boundary, and will allow us to state discrete analogues of the multisymplectic form formula and Noether's theorem. We refer to it as ``weak", since its definition involves integration and it is not a pointwise-defined quantity. Note that, unlike the weak Cartan form \eqref{Weak Cartan Form}, which required the higher regularity assumption $\partial_3\mathcal{L}(j^1_d\phi) \in H^1$, the above makes sense even when $\phi, v \in H\Lambda^k$. However, if the finite element subspace does have enough regularity to make sense of the pairing of the traces on the boundary, the above can be rewritten as
\begin{align*}
\delta S_U[\phi_h]\cdot v &= \int_{\partial U}  v \wedge * \partial_3\mathcal{L}(j^1\phi_h) + \sum_{T \in \mathcal{T}[\partial U]} \int_T (\partial_2\mathcal{L}(j^1\phi_h) + d^*\partial_3\mathcal{L}(j^1\phi_h))\wedge\star v_{\partial}.
\end{align*}

\begin{definition}[Discrete Weak Cartan Form]\label{Discrete Cartan Form}
The discrete weak Cartan form on a regular region $U$, evaluated at a field $\phi \in \Lambda^k_h$ and a variation $v$, is defined by
\begin{equation}\label{Discrete Cartan Form 1}
\Theta^h_U(\phi) \cdot v \equiv \delta S_U[\phi]\cdot v_{\partial}.
\end{equation}
\end{definition}

\begin{remark}\label{Discrete Cartan Form Remark}
Analogous to our discussion of the weak Cartan form in Section \ref{Weak Lagrangian Field Theory Section}, we can instead think of the discrete weak Cartan form as acting on vector fields. We make identification $T\Lambda^k_h \cong \Lambda^k_h \times \Lambda^k_h$, so that a vector field can be viewed as a map $V: \Lambda^k_h \rightarrow \Lambda^k_h$. Hence, the action of the discrete weak Cartan form on $V$ can be expressed as
$$ \mathbf{\Theta}^h_U (\phi) \cdot V = \Theta^h_U(\phi)\cdot V(\phi). $$
This identification will be useful when we prove the discrete multisymplectic form formula, since we will view first variations as vector fields on $\Lambda^k_h$ whose flow preserves the DEL. 
\end{remark}

Even though the weak Cartan form only involves integration on $\partial U$ whereas the discrete weak Cartan form involves integration on $\partial U$ and over regions $T \in \mathcal{T}[\partial U]$, this is the appropriate definition in the discrete setting because it encodes the boundary variation of the action, i.e., it equals the variation of the action when the discrete Euler--Lagrange equations are imposed. 


\begin{remark}[Quadrature]\label{Remark on Quadrature}
Although, in our exposition, we have assumed that with the given Lagrangian and choice of finite element space, one can evaluate the integrals involved exactly, one can more generally utilize quadrature to approximate the action before enforcing the variational principle. For a regular region $U$, let us consider quadrature nodes $\{c_a \in U\}$ and associated quadrature weights $\{b_a\}$. With finite element shape functions $\{v_j\}$ and expressing the density as $\mathcal{L} = L d^{n+1}x$, the associated discrete action is given by applying quadrature,
\begin{equation}\label{Discrete Action with Quadrature}
\mathbb{S}_U[\{\phi^j\}] = \sum_a b_a L(j^1_d(\phi^iv_i))|_{c_a}.
\end{equation}
The variation in the direction $w = w^kv_k$ is given by
\begin{equation}\label{Discrete Variation with Quadrature}
\delta \mathbb{S}_U[\{\phi^j\}]\cdot \{w^k\} = \sum_a b_a \frac{\partial}{\partial \phi^k} \Big[ L(j^1_d(\phi^iv_i))\big|_{c_a} \Big] w^k.
\end{equation}
The associated discrete Euler--Lagrange equation is given by enforcing the variational principle for variations $w$ with vanishing trace on $\partial U$. Then, the discrete Cartan form with quadrature (at a solution of the discrete Euler--Lagrange equation), is defined by taking an arbitrary variation and removing the term on the interior which vanishes by the discrete Euler--Lagrange equation. In particular, it is given by summing over all $a$ such that $c_a$ is contained in the support of some shape function with nonvanishing trace on the boundary; we denote the set of all such $a$ by $\mathcal{I}[\partial U]$. Hence, the discrete Cartan form with quadrature is given by 
$$ \Uptheta^h_U(\phi)\cdot w = \sum_{a \in \mathcal{I}[\partial U]} b_a \frac{\partial}{\partial \phi^k} \Big[ L(j^1_d(\phi^iv_i))\big|_{c_a} \Big] w^k. $$

Using this discrete Cartan form, an analogous statement of discrete multisymplecticity that we state below holds in this setting, with the caveat that the first variations are defined relative to the discrete Euler--Lagrange equations with quadrature. Similarly, an analogous statement to the discrete Noether's theorem below also holds in this setting, with the caveat that the group action leaves the discrete action with quadrature, equation (\ref{Discrete Action with Quadrature}), invariant. This is a direct consequence of the fact that the formulation with quadrature is still variational, since we applied the quadrature rule to the action, before enforcing the variational principle (see Section \ref{Variational Complex Section}). In general, if one applies quadrature after enforcing the variational principle, i.e., to the equations of motion (\ref{DEL 1b}), the system is not variational. To see this, we compute the variation of the action first,
$$ \delta S_U[\phi^jv_j]\cdot (w^kv_k) = \int_X [\partial_2\mathcal{L}(j^1_d(\phi^iv_i)) + d^*\partial_3\mathcal{L}(j^1_d(\phi^iv_i))]\wedge \star w^kv_k, $$
(for $w$ with vanishing trace on $\partial U$) and subsequently apply quadrature, so that the above becomes
$$ \sum_a b_a\Big[*\Big( [\partial_2\mathcal{L}(j^1_d(\phi^iv_i)) + d^*\partial_3\mathcal{L}(j^1_d(\phi^iv_i))]\wedge \star w^kv_k \Big) \Big]\Big|_{c_a}. $$
In general, this is not equal to (\ref{Discrete Variation with Quadrature}), except when $\phi$ a scalar field, using nodal interpolating shape functions and quadrature points at those nodes, in which case they are the same. Thus, for a variational formulation, one should generally apply quadrature before enforcing the variational principle. For the rest of the paper, we will revert to the assumption that one can evaluate the various integrals exactly, but keeping in mind that similar results hold in the case of quadrature. 
\end{remark}

We make several additional remarks regarding this candidate (\ref{Discrete Cartan Form 1}) for a discrete Cartan form. We defined the discrete Cartan form as the variation of the action, for variations that may be nonvanishing on the boundary, at a solution of the discrete Euler--Lagrange equations. Even though this functional involves integration over top-dimensional regions $T \in \mathcal{T}[\partial U]$, it only depends on the degrees of freedom which contribute to the nonzero value of $V$ on $\partial U$ and so makes sense as a candidate for a discrete Cartan form. In the continuum variational problem, boundary variations can be supported arbitrarily close to $\partial U$, whereas in the finite element variational problem, this is not the case, so the discrete Cartan form, which encodes the contribution of the variation of the action by boundary variations, should indeed contain the additional terms involving integration over the elements of $\mathcal{T}[\partial U]$. These terms shrink relative to the integral over $\partial U$ in the following heuristic sense. The terms involving $\mathcal{T}[\partial U]$ are $O(h)$ smaller than the term over $\partial U$: the cardinality of $\mathcal{T}[\partial U]$ scales like the number of boundary faces in $\partial U$, which is $O(h^{-n})$; on the other hand, the size of $T$ is $O(h^{n+1})$, so the terms in the discrete Cartan form involving the sum over $\mathcal{T}[\partial U]$ is $O(h)$, whereas the first term is $O(1)$ for a fixed region $U$. Thus, as $h \rightarrow 0$, for a fixed region $U$, the Cartan form formally only involves the first contribution, as expected. In other words, as we refine the mesh, $\partial U$ stays (roughly) the same, while the region containing only elements touching $\partial U$ shrinks, and a similar remark applies to the discrete multisymplectic form formula and the additional terms involving the sum over $\mathcal{T}[\partial U]$, so that the multisymplectic form formula is formally recovered in the limit. This can be combined with bounds on the integrands to show convergence more rigorously. More precisely, to show that the discrete weak Cartan form converges to the weak Cartan form, we would aim to show that for a solution $\phi_h$ of the DEL and a solution $\phi$ of the weak Euler--Lagrange equations and a variation $v$,
$$ \left|\Theta^h_U(\phi_h)\cdot\pi_hv - \Theta_U(\phi)\cdot v\right| \rightarrow 0, $$
as $h \rightarrow 0^+$. This error can be decomposed as
\begin{align*}
\left|\Theta^h_U(\phi_h)\cdot\pi_hv - \Theta_U(\phi)\cdot v\right| &\leq \left|\left(\Theta^h_U(\phi_h) - \Theta^h_U(\pi_h\phi)\right) \cdot \pi_hv\right| \\ &\quad + \Big|\Theta_U(\phi)\cdot (v-\pi_hv)\Big| + \left|\left(\Theta_U(\phi)-\Theta^h_U(\pi_h\phi)\right)\cdot\pi_hv\right|.
\end{align*}
The first term on the right hand side can be shown to converge with an appropriate quasi-optimatility bound between the discrete solution $\phi_h$ and the projected weak solution $\pi_h\phi$, and assuming the projection $\pi_h$ is bounded. The second term can be shown to converge if $\|v-\pi_hv\|$ converges in some appropriate norm. The third term can be shown to converge with an appropriate quasi-optimality bound between the projected weak solution and the weak solution, and again assuming the projection is bounded. Here, ``appropriate" qualifies the fact that the above terms involve the Cartan form which is defined in terms of derivatives of the Lagrangian density and hence, will be dependent on the particular theory under consideration. We provide an example in Section \ref{Numerical Example Section}.

We now show that Definition \ref{Discrete Cartan Form} recovers the notion of the discrete Cartan form introduced in \citet{MaPaSh1998} and further examined in \citet{Ch2008}, in the case that the degrees of freedom are the nodal values of the field with nodal interpolating shape functions. As previously remarked, in this case, the shape functions which are nonzero on $\partial U$ are those associated to nodes on $\partial U$. Consider a single node $i$ on $\partial U$ and let $v_i$ be the shape function associated to the degree of freedom on the node. Note that $v_i$ (restricted to $U$) is supported in some $T_i \in \mathcal{T}[\partial U]$ and denote $F_i = \partial T_i \cap \partial U$. Consider a variation of the form $V_i v_i\ (V_i \in \mathbb{R})$. \citet{MaPaSh1998} and \citet{Ch2008} define the discrete Cartan form associated to this node as $\frac{\delta S_U[\phi^jv_j]}{\delta \phi^i} V_i$ (no summation over $i$), viewing the action as a function of the components in the expansion of $\phi = \phi^jv_j$. Then, compute
\begin{align*}
\frac{\delta S_U[\phi^jv_j]}{\delta \phi^i} V_i &= \int_U \Big( \frac{\delta \mathcal{L}}{\delta \phi} \cdot \frac{\delta (\phi^jv_j)}{\delta \phi^i} V_i + \frac{\delta \mathcal{L}}{\delta (d\phi)} \cdot \frac{\delta (\phi^j dv_j)}{\delta \phi^i} V_i \Big) \\ 
&= \int_U \Big( \frac{\delta \mathcal{L}}{\delta \phi} \cdot V_iv_i + \frac{\delta \mathcal{L}}{\delta (d\phi)} \cdot V_i dv_i \Big) = \int_{U} \Big( \partial_2 \mathcal{L} \wedge \star V_iv_i + \partial_3\mathcal{L} \wedge \star V_idv_i \Big).
\end{align*}
Summing over all such variations on each node on $\partial U$, one recovers our discrete Cartan form, equation \eqref{Discrete Cartan Form 1}. There are several generalizations which our discrete Cartan form makes relative to the discrete Cartan form of \citet{MaPaSh1998} and \citet{Ch2008}. First, note that their Cartan form is defined in terms of the nodal values of the field, which implicitly suppresses the fact that the Cartan form involves integration over both $\partial U$ and elements of $\mathcal{T}[\partial U]$. Our explicit formula for the discrete Cartan form lends itself more easily to showing convergence to the continuum Cartan form, as we sketched heuristically above and will discuss further when discussing Noether's theorem. That the discrete Cartan form involves integration over elements neighboring the boundary is inevitable, since a variation of the field value on the boundary induces changes to the field values on elements of $\mathcal{T}[U]$. Furthermore, since we allow for general finite element spaces, we immediately obtain several generalizations. First, note that the dimension of the spacetime is arbitrary in our formulation, so this discrete Cartan form holds beyond the $1+1$ spacetime dimensions that they utilize explicitly in their framework (although this is not a fundamental restriction in their theory). Furthermore, our framework allows for differential forms of arbitrary degree, as opposed to just scalar fields. In particular, the degrees of freedom associated to the boundary variations need not be nodal values, but can be determined by more general degrees of freedom, such as moments or flux type degrees of freedom, e.g., when considering a theory involving vector fields, which one can identify with $1-$forms via the metric. Furthermore, these degrees of freedom determining the boundary variations may be close to, i.e., in $\mathcal{T}[\partial U]$, but not necessarily on $\partial U$. 

In the next two sections, we will utilize the discrete Cartan form to state discrete analogues of multisymplecticity and Noether's theorem. We will see that these statements, involving $\Theta^h_U$, will be in direct analogy to the continuum theorems, involving $\Theta_U$.

\subsection{Discrete Multisymplectic Form Formula}\label{Multisymplectic Section}

We now state a discrete analogue of the multisymplectic form formula, which generalizes the preservation of the symplectic form under the flow of a symplectic vector field. In the smooth setting, if $\phi$ is a solution to the Euler--Lagrange equations and $V, W$ are first variations at $\phi$, i.e., their respective flows on $\phi$ are still solutions, then 
\begin{equation}\label{Multisymplectic Form Formula}
\int_{\partial U}(j^1\phi)^*\Big( j^1V \lrcorner\, j^1W\lrcorner\, \Omega_{\mathcal{L}} \Big) = 0, 
\end{equation}
where $U \subset X$ is a submanifold with smooth closed boundary (\citet{MaPaSh1998}). The multisymplectic form formula encompasses many physical conservation laws appearing in Lagrangian field theories. For example, viewing a Lagrangian field theory in the instantaneous canonical formulation, multisymplecticity gives rise to the usual field-theoretic notion of symplecticity (\citet{MaPaSh1998}). Furthermore, multisymplecticity encompasses the notion of reciprocity in many physical systems, relating the infinitesimal perturbation of a system by a source and the associated infinitesimal perturbation of the response by the system (see, for example, \citet{VaLiLe2011} for Lorenz reciprocity in electromagnetism and \citet{McAr2020} for reciprocity in semilinear elliptic PDEs, within the context of multisymplecticity). Additionally, for wave propagation problems, multisymplecticity provides a geometric formulation for the conservation of wave action (\citet{Br1997, Br1997no2}). Since multisymplecticity is an important property of Lagrangian field theories encompassing many natural physical conservation laws, we will investigate multisymplecticity within our discretization framework. 

In the literature, integrators which admit a discrete analogue of this formula are referred to as ``multisymplectic integrators''. We show that our discrete system (\ref{DEL 1b}) admits a discrete multisymplectic form formula. The main idea of the derivation for the multisymplectic form formula is to look at second variations of the action at $\phi$ with respect to first variations $V$ and $W$, $d^2S[\phi]\cdot (V,W) = 0$. More specifically, one decomposes the variation of the action into two functionals, corresponding to interior and boundary variations:
\begin{align*}
dS[\phi]\cdot V = \underbrace{-\int_U (j^1\phi)^* (j^1V \lrcorner\, \Omega_{\mathcal{L}})}_{\equiv \text{EL}_U(\phi)\cdot V} + \underbrace{\int_{\partial U}(j^1\phi)^*(j^1V \lrcorner\, \Theta_{\mathcal{L}})}_{=\Theta_U(\phi)\cdot V}.
\end{align*}
Then, $0 = d^2S[\phi]\cdot (V,W) = d\text{EL}_U(\phi)\cdot (V,W) + d\Theta_U(\phi)\cdot (V,W)$. The term $d\text{EL}_U(\phi)\cdot (V,W)$ vanishes from the first variation property, so the multisymplectic form formula can be expressed as
$$ d\Theta_U(\phi)\cdot (V,W) = 0, $$
which is equivalent to equation (\ref{Multisymplectic Form Formula}).

In our construction, the first difference is that we are working in the weak setting, as discussed in Section \ref{Weak Lagrangian Field Theory Section}. Furthermore, in the discretized theory, the main impediment for a discrete analogue of the multisymplectic form formula is that a solution of the discrete equation (\ref{DEL 1b}) does not in general satisfy an Euler--Lagrange equation locally (i.e., for arbitrary $U$) but rather integrated over a regular region $U$. Additionally, there is an additional contribution from the boundary components of the variation in the elements neighboring the boundary $T \in \mathcal{T}[\partial U]$. It is in this restricted setting that we have a discrete multisymplectic form formula.

To prepare for the proof of the discrete multisymplectic form formula, we will express variations in terms of vector fields. As briefly discussed in Remark \ref{Discrete Cartan Form Remark}, we can express the action of the Cartan form in terms of vector fields, instead of variations, which follows from the identification $T(\Lambda^k_h) \cong \Lambda^k_h \times \Lambda^k_h$ and thus, a vector field $V \in \mathfrak{X}(\Lambda^k_h)$ can be viewed as a map $V: \Lambda^k_h \rightarrow \Lambda^k_h$. Similarly, for a vector field $V$, the variation of the action can be expressed as
$$ dS[\phi]\cdot V = \delta S [\phi] \cdot V(\phi). $$
Furthermore, we decompose a vector field into its interior and boundary components,
\begin{align*}
V_{in}&: \phi \mapsto (V(\phi))_{in}, \\
V_{\partial}&: \phi \mapsto (V(\phi))_{\partial}.
\end{align*}

\begin{theorem}{\textbf{(Discrete Multisymplectic Form Formula)}}\label{Discrete Multisymplectic Form Theorem}
Let $U$ be a regular region and let $\phi_h$ be a solution of the local DEL (\ref{local DEL}) and $V, W \in \mathfrak{X}(\Lambda^k_h)$ be first variations for $\phi_h$, i.e., their flow on $\phi_h$ still satisfies the DEL, but for arbitrary boundary variations, then
\begin{equation}\label{Discrete Multisymplectic Form Formula 1}
d\mathbf{\Theta}^h_U(\phi_h)\cdot (V,W) = 0.
\end{equation}
\begin{proof}
Decompose the variation of the action into interior and boundary variations,
$$ dS_U[\phi_h]\cdot V = \underbrace{dS_U[\phi_h] \cdot V_{in}}_{\equiv \textbf{EL}^h_U(\phi_h)\cdot V} + \underbrace{dS[\phi_h] \cdot V_{\partial}}_{= \mathbf{\Theta}^h_U(\phi_h) \cdot V}. $$
so that $dS[\phi_h]\cdot V = \textbf{EL}^h_U(\phi_h)\cdot V + \mathbf{\Theta}^h_U(\phi_h) \cdot V$. Observe that, by definition of $\textbf{EL}^h_U$, the DEL is equivalent to the statement that $\textbf{EL}^h_U(\phi_h) \cdot V = 0$ for all $V \in \mathfrak{X}(\Lambda^k_h)$. Thus, we define a first variation $W$ as a vector field which preserves the DEL, $d(\textbf{EL}^h_U(\phi_h) \cdot V)\cdot W = 0$. Thus, we have
$$0 = d^2S_U[\phi_h]\cdot (V,W) = d\textbf{EL}^h_U(\phi_h)\cdot (V,W) + d\mathbf{\Theta}^h_U(\phi_h)\cdot (V,W).$$
Then, express
$$ d\textbf{EL}^h_U(\phi_h)\cdot (V,W) = d( \textbf{EL}^h_U(\phi_h)\cdot V)\cdot W - d(\textbf{EL}^h_U(\phi_h)\cdot W)\cdot V - \textbf{EL}^h_U(\phi_h)\cdot [V,W]. $$
The first two terms on the right hand side of the above equation vanish by the definition of first variation; furthermore, the third term vanishes by the DEL. Hence, $d\textbf{EL}^h_U(\phi_h)\cdot (V,W) = 0$. Thus, we have
$$ d\Theta^h_U(\phi_h)\cdot (V,W) = d^2S_U[\phi_h]\cdot (V,W) = 0. $$
\end{proof}
\end{theorem}

\begin{remark}
Although we immediately see that the discrete multisymplectic form formula $d\mathbf{\Theta}^h_U(\phi_h)\cdot (V,W) = 0$ is in direct analogy with the continuum multisymplectic form formula $d\Theta_U(\phi)\cdot (V,W) = 0$, if we write the discrete formula using the definition of the discrete Cartan form, we see that there is an additional contribution corresponding to the integration over elements $T \in \mathcal{T}[\partial U]$. Although we will not write this out explicitly, we see that this additional contribution involves a sum-integral of the form $\sum_{T \in \mathcal{T}[\partial U]} \int_T$, which is $O(h)$ as discussed previously. As such, we only need control of the residual associated to the linearized equations to formally show convergence of the discrete multisymplectic form formula to the continuum multisymplectic form formula.

We note that the aforementioned convergence is formal since it must also be combined appropriately with convergence of the discrete solution to a continuum weak solution using bounds on the projection. One possible method for combining these is the following observation. Since, by assumption, the projections are cochain projections, we have that
$$ d^2S_h[\phi] \cdot (V,W) = d^2(\pi_h^*S)[\phi]  \cdot (V,W) = d^2S[\pi_h\phi] \cdot (T\pi_h \cdot V, T\pi_h \cdot W). $$
In particular, for first variations $V,W \in \mathfrak{X}(Y)$ for the degenerate action, $T\pi_h \cdot V, T\pi_h\cdot W$ correspond to first variations of the discrete Euler--Lagrange equations, and the discrete multisymplectic form formula can be reinterpreted as the multisymplectic form formula for the degenerate action. Note also that for cochain projections, a simple calculation shows that $j^1( T\pi_h \cdot V) = T(\pi_h^k \times \pi_h^{k+1}) \cdot j^1V$, so that the terms in the integrand of the discrete multisymplectic form formula, (\ref{Discrete Multisymplectic Form Formula 1}), are in the image of the (tangent) projections. This allows us to formulate the discrete multisymplectic formula in terms of the projection and its tangent lift, and hence more directly determine in what sense the discrete multisymplectic form formula converges as $h \rightarrow 0$.

Of course, without specifying a particular field theory and finite element spaces, we cannot proceed further to show convergence. We aim to investigate more rigorous convergence results for particular field theories in future work. See also the discussion below regarding convergence of the discrete Noether theorem to its continuum analogue. 

\end{remark}

\begin{remark}
As noted before, the discrete Cartan form, in the case of nodal interpolating shape functions, gives precisely the discrete notion of Cartan form introduced in \citet{MaPaSh1998}. In this case, our discrete multisymplectic form formula $d\mathbf{\Theta}^h_U(\phi_h)\cdot (V,W) = 0$ (for first variations $V,W$) gives precisely the discrete multisymplectic form formula derived in \citet{MaPaSh1998}.
\end{remark}

\subsection{Discrete Noether's Theorem}\label{Noether's Theorem Section}
In this section, we establish a discrete analogue of the weak Noether's theorem as discussed in Section \ref{Weak Lagrangian Field Theory Section}.

To derive a discrete analogue, we first must restrict to regular regions instead of allowing arbitrary regions, analogous to the discussion of the discrete multisymplectic form formula. Furthermore, we must make sense of a group action on the discrete space $\Lambda^k_h \subset H\Lambda^k$. In general, one cannot expect the group action $G \times H\Lambda^k \rightarrow H\Lambda^k$ to restrict to a group action $G \times \Lambda^k_h \rightarrow \Lambda^k_h$, i.e., the group orbit $G \cdot \Lambda^k_h$ is not necessarily contained in $\Lambda^k_h$. However, suppose there exists a Lie group homomorphism $\psi_h: G \rightarrow G $ such that
$$ \psi_h(g)\cdot \pi_h\phi = \pi_h (g \cdot \phi), \text{ for all } g \in G, \phi \in H\Lambda^k. $$
In such a case, we say that the projection $\pi_h$ is $G$-equivariant with intertwining homomorphism $\psi_h$. 
\begin{remark}
In essence, the motivation behind this definition is that when one discretizes a theory, a symmetry group of the original theory may be reduced to a smaller subgroup. This is encoded in the homomorphism $\psi_h$, where the smaller subgroup is $\psi_h(G) \cong G/\ker{\psi_h}$. We will see some examples of this after proving a discrete Noether's theorem.
\end{remark}

We are now ready to state a discrete Noether's theorem.

\begin{theorem}
Let $U$ be a regular region. Suppose the action $S$ is $G$-invariant and the projection is $G$-equivariant with intertwining homomorphism $\psi_h$. Then, for a solution $\phi_h \in \Lambda^k_h$ of the DEL,
\begin{subequations}
\begin{equation}\label{Discrete Noether Theorem 1a}
\mathbf{\Theta}^h_U(\phi_h) \cdot T\pi_h(\widetilde{\xi}) = 0,
\end{equation}
or, equivalently,
\begin{equation}\label{Discrete Noether Theorem 1b}
\mathbf{\Theta}^h_U(\phi_h) \cdot \widetilde{(\psi_h)_*\xi} = 0, 
\end{equation}
\end{subequations}
where $(\psi_h)_*$ is the induced Lie algebra homomorphism.
\begin{proof}
Since $\phi_h \in \Lambda^k_h$ and $\pi_h: H\Lambda^k \rightarrow \Lambda^k_h$ is surjective, there exists some $\phi \in H\Lambda^k$ such that $\phi_h = \pi_h\phi$. Then, for any $g \in G$, 
$$ S_U[\pi_h(g \cdot \phi)] = S_U[ \psi_h(g) \cdot \pi_h\phi] = S_U[\pi_h\phi], $$
where $G$-equivariance of the projection was used in the first equality and $G$-invariance of the action was used in the second equality. The above holds for all $g$ and in particular, for $\xi \in \text{Lie}(G)$, one has that
$$ S_U \circ \pi_h [e^{t\xi} \cdot \phi] = S_U \circ \pi_h[\phi]. $$
Differentiating the above yields
$$ 0 = d(S_U \circ \pi_h)[\phi] \cdot \widetilde{\xi} = \pi_h^* dS_U[\phi]\cdot \widetilde{\xi} = dS_U[\pi_h\phi]\cdot T\pi_h \widetilde{\xi} = dS_U[\phi_h]\cdot T\pi_h \widetilde{\xi}.$$
Finally, we decompose the variation of the action as
$$ 0 = dS_U[\phi_h]\cdot T\pi_h \widetilde{\xi} = \textbf{EL}^h_U(\phi_h) \cdot T\pi_h \widetilde{\xi} + \mathbf{\Theta}^h_U(\phi_h)\cdot T\pi_h \widetilde{\xi}.$$
The term $\textbf{EL}^h_U(\phi_h) \cdot T\pi_h \widetilde{\xi}$ vanishes since $\phi_h$ satisfies the DEL and $T\pi_h \widetilde{\xi}$ is a vector field on $\Lambda^k_h$. Thus, equation \eqref{Discrete Noether Theorem 1a} follows. 

To see that this is equivalent to equation \eqref{Discrete Noether Theorem 1b}, it suffices to show $T\pi_h \widetilde{\xi}(\phi) = \widetilde{ (\psi_h)_*\xi} (\phi_h)$. To see this, recall that
$$ \widetilde{\xi}(\phi) = \lim_{t\rightarrow 0} \frac{e^{t\xi} \cdot \phi - \phi}{t}. $$
Thus, the pushforward can be computed as
\begin{align*}
T\pi_h \widetilde{\xi}(\phi) &= \lim_{t\rightarrow 0} \frac{\pi_h (e^{t\xi} \cdot \phi) - \pi_h \phi}{t} \\
&= \lim_{t\rightarrow 0} \frac{\psi_h(e^{t\xi}) \cdot \pi_h\phi - \pi_h\phi}{t} \\
&= \widetilde{(\psi_h)_*\xi},
\end{align*}
where $G$-equivariance was used in the second equality and the third equality is simply the definition of an infinitesimal generator. 
\end{proof}
\end{theorem}

\begin{remark}
Note that the proof above is still valid if one weakens the notion of $G$-equivariance to only hold infinitesimally up to $o(t)$, i.e.,
$$ \psi_h(e^{t\xi}) \cdot \pi_h\phi = \pi_h(e^{t\xi}\cdot \phi) + o(t), \text{ for all } \xi \in \text{Lie}(G), \phi \in H\Lambda^k. $$
\end{remark}

We give two simple examples of group-equivariant cochain projections and subsequently remark on how one might construct more general group-equivariant cochain projections.

\begin{example}[Global Linear Group Action]
First, note that although we took our field configuration bundle to be $\Lambda^k(X)$, we could have more generally taken our fields to be vector-valued forms, corresponding to the bundle $\Lambda^k(X) \otimes V$ for some finite-dimensional vector space $V$. With a basis $\{e_i\}$ for $V$, the only modification to the discrete Euler--Lagrange (\ref{DEL 1b}) equation is that there are $\dim(V)$ equations corresponding to each component of the field $\phi^i \in \Lambda^k(X)$ in the expansion $\phi(x) = \sum_i \phi_i(x) \otimes e_i$.

Suppose that a Lagrangian with such a configuration bundle is invariant under the global action by a group representation $D: G \rightarrow GL(V)$. That is, $D$ acts on $\phi \in \Lambda^k(X)\otimes V$ as $1_{\Lambda^k(X)} \otimes D$:
$$ D(g)\phi(x) = \sum_i\phi_i(x) \otimes ( D(g)e_i ), $$
where $D(g)$ is independent of $x$.

Let $\pi^k_h: H\Lambda^k \rightarrow \Lambda^k_h$ and $\pi^{k+1}_h: H\Lambda^{k+1} \rightarrow \Lambda^k_h$ be cochain projections, i.e., they satisfy $\pi^{k+1}_hd = d\pi^k_h$. We can extend these to cochain projections on vector-valued forms by $\tilde{\pi}_h = \pi_h \otimes 1_V$. Furthermore, group-equivariance follows from linearity of the group action and the above definitions,
\begin{align*}
D(g) \tilde{\pi}_h \phi &= D(g) \tilde{\pi}_h \left(\sum_i \phi_i \otimes e_i\right) = D(g) \sum_i \pi_h(\phi_i)\otimes e_i = \sum_i \pi_h(\phi_i) \otimes D(g)e_i \\
&= \tilde{\pi}_h \left( \sum_i \phi_i \otimes D(g) e_i \right) = \tilde{\pi}_h \left( D(g) \sum_i \phi_i \otimes e_i \right) = \tilde{\pi}_h D(g)\phi.
\end{align*}
Thus, the discrete Noether's theorem holds in this case, where the intertwining homomorphism is just the identity.

A simple example of such a theory is the Schr\"{o}dinger equation with $V = \mathbb{C}$, $G = U(1)$, and the group representation given by the fundamental representation of $U(1)$ in $GL(\mathbb{C})$. The corresponding Noether conservation law is conservation of mass in the $L^2$ norm.
\end{example}

\begin{example}[Yang--Mills Theory]
As an example of a non-global (but still linear) group action, consider Yang--Mills theories with a structure group $G$. In this setting, the field $A \in \Lambda^1(X) \otimes \mathfrak{g}$, i.e., $A$ is valued in the Lie algebra $\mathfrak{g}$ associated to $G$. More precisely, the field is valued in the adjoint representation of the Lie algebra. This class of theories is invariant under the linear action of $\Lambda^0(X) \otimes \mathfrak{g}$, viewed as a group under addition, on $\Lambda^1(X) \otimes \mathfrak{g}$ given by
$$ \alpha \cdot A \equiv A + d\alpha, $$
for any $\alpha \in \Lambda^0(X)\otimes \mathfrak{g}$. Unlike the previous example, this action is local in the sense that $D(\alpha)$ depends on the position in spacetime. 

Now, suppose that we have cochain projections for the sequence $H\Lambda^0 \overset{d}{\rightarrow} H\Lambda^1 \overset{d}{\rightarrow} H\Lambda^2$, i.e., $\pi^2_hd = d\pi^1_h, \pi^1_hd = d\pi^0_h$. Extend these to projections $\tilde{\pi}_h$ on $H\Lambda \otimes \mathfrak{g}$ as in the previous example. The relation $\tilde{\pi}^2_hd = d\tilde{\pi}^1_h$ is required for naturality of the variational structure. On the other hand, the relation $\tilde{\pi}^1_hd = d\tilde{\pi}^0_h$ gives group equivariance in the following sense,
$$ \tilde{\pi}^1_h( \alpha \cdot A ) = \tilde{\pi}^1_h(A + d\alpha) = \tilde{\pi}^1_hA + \tilde{\pi}^1_h d\alpha = \tilde{\pi}^1_hA + d\tilde{\pi}^0_h \alpha = \tilde{\pi}^0_h (\alpha) \cdot \tilde{\pi}^1_h A. $$
Thus, the discrete Noether's theorem holds where the intertwining homomorphism is $\psi_h = \tilde{\pi}^0_h$.

In the continuum Hilbert space setting, the associated conservation law is the weak Gauss' law, where Gauss' law holds tested against any element of the Hilbert space. In the discrete setting, the discrete Noether's theorem gives a discrete Gauss' law, where Gauss' law holds tested against any element of the finite-dimensional subspace.  

\end{example}

The previous two examples were simple in the sense that they had a linear or global group action. Although the second example was local, the acting group is contained in the Hilbert complex of forms and group-equivariance arose from having cochain projections. 

To construct group-equivariant cochain projections for more general actions, one possible method would be to utilize group-equivariant interpolation \citep{GaLe2018,Le2019} in constructing the projection. One method to construct cochain projections from interpolants is to place an intermediate sequence between the sequence of Hilbert spaces and the sequence of finite-dimensional subspaces,
\[\begin{tikzcd}
	{H\Lambda^k} && {H\Lambda^{k+1}} \\
	{} \\
	{C^k} && {C^{k+1}} \\
	\\
	{\Lambda^k_h} && {\Lambda^{k+1}_h}
	\arrow["{\sigma^k}"', from=1-1, to=3-1]
	\arrow["d", from=1-1, to=1-3]
	\arrow["{\sigma^{k+1}}"', from=1-3, to=3-3]
	\arrow["{D}", from=3-1, to=3-3]
	\arrow["{\mathcal{I}^k}"', from=3-1, to=5-1]
	\arrow["d", from=5-1, to=5-3]
	\arrow["{\mathcal{I}^{k+1}}"', from=3-3, to=5-3],
\end{tikzcd}\]
where $\{\sigma^m\}$ are the degrees of freedom mapping into the coefficient spaces $\{C^m\}$, $\{\mathcal{I}^m\}$ are interpolants from the coefficient spaces into the finite-dimensional subspaces, $D$ realizes $d$ in the coefficient space, and the projections are defined by $\pi_h = \mathcal{I} \circ \sigma$. The degrees of freedom must be unisolvent when restricted to the image of the interpolants. Constructing cochain projections amounts to ensuring that the top diagram commutes. Then, fixing group-equivariant interpolants $\mathcal{I}^k, \mathcal{I}^{k+1}$, group-equivariant cochain projections could be achieved by choosing the degrees of freedom such that they are unisolvent for this choice of interpolants and ensuring that the top diagram commutes. We will pursue such a construction in future work.

\subsection{A Discrete Variational Complex}\label{Variational Complex Section}
The variational bicomplex is a double complex on the spaces of differential forms over the jet bundle of a configuration bundle used to study the variational structures of Lagrangian field theories defined on this bundle (see, for example, \citet{An1992}). The differential forms arising in Lagrangian field theory, such as the Lagrangian density, the Cartan form, and the multisymplectic form, can be interpreted as elements of this variational bicomplex. The cochain maps in this double complex are the horizontal and vertical exterior derivatives on the jet bundle, which give a geometric interpretation to the variations encountered in Lagrangian field theories. The variational bicomplex has also been extended to problems with symmetry in \citet{KoOl2003}, and to the discrete setting for difference equations corresponding to discretizing Lagrangian field theories on a lattice in \citet{HyMa2004}.

In this section, we interpret and summarize the results from the previous sections in terms of a discrete variational complex which arises naturally in our discrete construction and, in a sense, resembles the vertical direction of the variational bicomplex. 

In our previous discussion, we saw a complex which arises from the space of discrete forms,
$$ \Lambda^0_h \overset{d}{\longrightarrow} \Lambda^1_h \overset{d}{\longrightarrow} \cdots \overset{d}{\longrightarrow} \Lambda^{n}_h \overset{d}{\longrightarrow} \Lambda^{n+1}_h, $$
which forms a complex due to the cochain projection property. Now, consider instead the following ``vertical" complex; consider the spaces of smooth forms on $\Lambda^k_h$, which we denote $\Omega (\Lambda^k_h)$, with the ``vertical" exterior derivative $d_v: \Omega^m(\Lambda^k_h) \rightarrow \Omega^{m+1}(\Lambda^k_h)$ being the usual exterior derivative over the base manifold $\Lambda^k_h$ (which is a vector space). This gives a discrete variational complex:
\begin{figure}[h]\label{Discrete Variational Complex}
\[\begin{tikzcd}
	{\Omega^{\dim(\Lambda^k_h)}(\Lambda^k_h)} \\
	\vdots \\
	{\Omega^1(\Lambda^k_h)} \\
	{\Omega^0(\Lambda^k_h)\ .}
	\arrow["{d_v}", from=4-1, to=3-1]
	\arrow["{d_v}", from=3-1, to=2-1]
	\arrow["{d_v}", from=2-1, to=1-1]
\end{tikzcd}\]
\end{figure}

Note that in the previous sections, we used $d$ to denote both the exterior derivative corresponding to the de Rham complex and the vertical exterior derivative, e.g., the multisymplectic form formula $d\mathbf{\Theta}^h_U(V,W) = 0$ is more precisely $d_v\mathbf{\Theta}^h_U(V,W) = 0$, where it was understood which was meant by the spaces where the relevant quantities were defined. However, we will distinguish the two in this section to be more precise. We call the above a vertical complex for two reasons: first, the vertical exterior derivative corresponds to differentiation with respect to the fiber values as we will see below. Furthermore, it resembles the vertical direction of the variational bicomplex. However, in our construction, there is no horizontal direction, since in the discrete setting, we are considering transgressed forms, i.e., forms integrated over a region, so the horizontal direction collapses.  

Examples of forms in the discrete variational complex include the restricted action $S \in \Omega^0(\Lambda^k_h)$, the discrete weak Cartan form $\mathbf{\Theta}^h \in \Omega^1(\Lambda^k_h)$, and the discrete multisymplectic form $d_v\mathbf{\Theta}^h \in \Omega^2(\Lambda^k_h)$. Let $\{v_i\}$ be a basis for $\Lambda^k_h$; we then coordinatize the vector space $\Lambda^k_h$ by the components of the expansion of any $\phi = \sum_i \phi^i v_i \in \Lambda^k_h$, which we denote as a vector $(\phi^i) = (\phi^0,\dots,\phi^{\dim(\Lambda^k_h)}) \in \Lambda^k_h$. For example, the vertical exterior derivative of the action is
$$ d_v S[\phi] = \sum_j \frac{\partial S[(\phi^i)]}{\partial \phi^j} d_v\phi^j. $$
The naturality of the variational principle and the interpretation of the weak Euler--Lagrange equations as a Galerkin variational integrator, discussed in Section \ref{Variational Structure Section}, relate the vertical exterior derivative of $S$ to the variation of the degenerate action $S_h$. Now, let $\Pi_i$ be the projection onto the $i^{th}$ coordinate $\phi^i$ and let $\mathcal{I}[\partial U]$ denote the set of indices $i$ such that $v_i$ has nonvanishing trace on $\partial U$. Then, for $v = (v^i) \in \Lambda^k_h$, we have that
\begin{align*}
v_{\partial} &= \sum_{i \in \mathcal{I}[\partial U]}\Pi_i (v), \\
v_{in} &= v - v_{\partial} = \sum_{i \not\in \mathcal{I}[\partial U]} \Pi_i (v).
\end{align*}
Recall that we can view vector fields $V \in \mathfrak{X}(\Lambda^k_h)$ as maps $V: \Lambda^k_h \rightarrow \Lambda^k_h$, and we extend this to the vector fields $V_{\partial}(\phi) \equiv (V(\phi))_{\partial}$ and $V_{in}(\phi) \equiv (V(\phi))_{in}$. In particular, the discrete weak Cartan form in this notation is given by
$$ \mathbf{\Theta}^h(\phi)\cdot V = d_vS[\phi]\cdot V_{\partial}. $$
The variation of the action can then be expressed as
$$ d_vS[\phi]\cdot V = \textbf{EL}^h(\phi)\cdot V + \mathbf{\Theta}^h(\phi)\cdot V. $$
More explicitly, these can be expressed as
\begin{align*}
\mathbf{\Theta}^h(\phi) &=  \sum_{j \in \mathcal{I}[\partial U]} \frac{\partial S[(\phi^i)]}{\partial \phi^j} d_v\phi^j, \\
\textbf{EL}^h(\phi) &=  \sum_{j \not\in \mathcal{I}[\partial U]} \frac{\partial S[(\phi^i)]}{\partial \phi^j} d_v\phi^j.
\end{align*}
In particular, the discrete Euler--Lagrange equations are given by the null Euler--Lagrange condition, $\textbf{EL}(\phi) = 0$, i.e., $\textbf{EL}(\phi)\cdot V = 0$ for all $V$. Assuming a solution $\phi$ of the null Euler--Lagrange condition, we immediately see that
$$ d_vS[\phi]\cdot V = \mathbf{\Theta}^h(\phi)\cdot V, $$
and in particular, for a symmetry of the action $d_vS[\phi]\cdot \tilde{\xi} = 0$, we have the discrete Noether's theorem $\mathbf{\Theta}^h(\phi)\cdot \tilde{\xi} = 0$. By taking the second exterior derivative of the action, we have that
$$ 0 = d_v^2 S[\phi] = d_v \textbf{EL}(\phi) + d_v \mathbf{\Theta}^h(\phi). $$
The space of first variations at $\phi$ is precisely the kernel of the quadratic form $d_v \textbf{EL}(\phi)$, so this gives the discrete multisymplectic form formula $d_v \mathbf{\Theta}^h(\phi)(\cdot,\cdot) = 0$ when evaluated on first variations. Thus, the results of the previous sections can be concisely summarized in terms of the structure given by the discrete variational complex. 

Furthermore, this framework also encompasses the discrete variational principle with quadrature, as discussed in Remark \ref{Remark on Quadrature}. Namely, from the discrete viewpoint, a discrete action is an element of $\Omega^0(\Lambda^k_h)$ and in particular, the discrete action with quadrature $\mathbb{S}$ from (\ref{Discrete Action with Quadrature}) is an element of $\Omega^0(\Lambda^k_h)$. Then, the variation of $\mathbb{S}$ can be decomposed into interior and boundary one-forms as before,
\begin{align*}
d_v \mathbb{S}[(\phi)] &= \mathbb{EL}(\phi) + \Uptheta^h(\phi), \\
\Uptheta^h(\phi) &=  \sum_{j \in \mathcal{I}[\partial U]} \frac{\partial \mathbb{S}[(\phi^i)]}{\partial \phi^j} d_v\phi^j, \\
\mathbb{EL}(\phi) &=  \sum_{j \not\in \mathcal{I}[\partial U]} \frac{\partial \mathbb{S}[(\phi^i)]}{\partial \phi^j} d_v\phi^j.
\end{align*}
The discrete Euler--Lagrange equations with quadrature are given by the null Euler--Lagrange condition $\mathbb{EL}(\phi) = 0$, and subsequently, the discrete Noether's theorem and discrete multisymplectic form formula (in the case of quadrature) then follow analogously to before, where symmetries are with respect to $\mathbb{S}$ and the space of first variations at $\phi$ is the kernel of the quadratic form $d_v\mathbb{EL}(\phi)$. 
\subsection{Numerical Example}\label{Numerical Example Section}
We consider the scalar Poisson equation in $(1+1)$-spacetime dimensions on a rectangular domain, $X = [a,b] \times [c,d]$,
$$ \partial_t^2 \phi + \partial_x^2\phi = f(x,y). $$
The Lagrangian is given by $L = \frac{1}{2} (\partial_t\phi)^2 + \epsilon \frac{1}{2}(\partial_x\phi)^2  + f(x,y)\phi$, or equivalently, the Lagrangian density is given by
$$ \mathcal{L} = \frac{1}{2} d\phi\wedge\star d\phi + f \wedge * \phi. $$
Compute $\partial_3\mathcal{L}(j^1_d\phi) = d\phi$, $\partial_2\mathcal{L}(j^1_d\phi) = f$, where we assume $f \in L^2\Lambda^0$, so the discrete Euler--Lagrange equation reads: find $\phi \in \Lambda^0_h$ such that
$$ (d\phi, dv)_{L^2} = (f,v)_{L^2}, \text{ for all } v \in \mathring{\Lambda}^0_h. $$
We subdivide $X$ into a regular rectangular mesh and use a tensor-product basis of hat functions $\psi_{ij}(t,x) = \chi_i(t) \xi_j(x) $ subordinate to this mesh.

Expressing $\phi = \phi^{ij} \psi_{ij}$ and taking $v = \psi_{mn}$, the above equation reads as
$$ \sum_{ij \in [mn]} \Big( \phi^{ij} (\chi_i'(t), \chi_m'(t))_{L^2} (\xi_j(x),\xi_n(x))_{L^2} + \phi^{ij}(\chi_i(t), \chi_m(t))_{L^2} (\xi_j'(x),\xi_n'(x))_{L^2}\Big) = (f, \psi_{mn})_{L^2}. $$
Since $[m] = \{m-1,m,m+1\},$ this gives a nine-point stencil on the interior elements of the mesh. Explicitly, we compute the stiffness and mass matrix elements
\begin{align*}
\{(\chi_i'(t),\chi_m'(t))_{L^2}\}_{i \in [m]} &= \frac{1}{\Delta t} \{-1,2,-1\}, \\
\{(\chi_i(t),\chi_m(t))_{L^2}\}_{i \in [m]} &= \Delta t \left\{ \frac{1}{6}, \frac{2}{3}, \frac{1}{6} \right\},
\end{align*}
and similarly for the $x$ direction. This gives 
$$ \frac{ \phi^{m+1 \tilde{n}} - 2 \phi^{m \tilde{n}} + \phi^{m-1 \tilde{n}}}{\Delta t^2} + \frac{\phi^{\tilde{m} n+1} - 2 \phi^{\tilde{m} n} + \phi^{\tilde{m} n-1} }{\Delta x^2} + \frac{1}{\Delta t \Delta x} (f,\psi_{mn}), $$
where $\phi^{m \tilde{n}} = \frac{1}{6}(\phi^{m n+1} +4 \phi^{mn} + \phi^{mn-1})$ and $\phi^{\tilde{m} n} = \frac{1}{6}(\phi^{m+1 n} +4 \phi^{mn} + \phi^{m-1 n})$. Noting that $(N'(\phi),\psi_{mn}) = \delta\mathcal{N}/\delta \phi^{mn}$, where $\mathcal{N} = \int N(\phi) dt\wedge dx$, this reproduces the nine-point variational integrator derived by \citet{Ch2008}. As was shown in \citet{Ch2008}, using mid-point quadrature, this method reduces to the multisymplectic integrator derived by \citet{MaPaSh1998}. 

Now, we consider the discrete Cartan form for this example. Consider a regular region $U \subset X$; for simplicity, we take $U$ to be a rectangular region $U = [t_0, t_M] \times [x_0, x_N]$, without loss of generality, since any regular region on a rectangular mesh is a union of such rectangular regular regions, where the vertices of $U$ are given by $\{(t_i,x_j)\}_{i,j=0}^{M,N}$ where $t_i = t_0 + i \Delta t, x_j = x_0 + j \Delta x$. We index the piecewise linear nodal interpolating shape functions $\psi_{ij}(t,x) = \chi_i(t)\xi_j(x)$ by the node $(t_i,x_j)$ which it interpolates, i.e., $\psi_{ij}(t_k,x_l)=\chi_i(t_k)\xi_j(x_l) = \delta_{ik}\delta_{jl}$. Let 
$$ \phi_h = \sum_{i,j=0}^{M,N}\phi^{ij}_h \psi_{ij} $$
be a solution of the associated discrete Euler--Lagrange equation, restricted to $U$.  

Recall the definition of the discrete weak Cartan form as the variation of the action by $w \in \Lambda^0_h(U)$, with generally nonvanishing trace on $\partial U$. Letting $w = w_{in} + w_{\partial} \in \Lambda^0_h(U)$ and $W \in \mathfrak{X}(\Lambda^0_h)$ such that $W(\phi_h) = w$, we have $\delta S_U[\phi_h]\cdot w_{in} = 0$ and hence,
\begin{align}\label{Discrete Cartan Form for Wave Eq}
 \Theta^h_U(\phi_h)\cdot W &= \delta S_U[\phi_h]\cdot w = \delta S_U[\phi_h]\cdot (w - w_{in}) = \delta S_U[\phi_h] \cdot w_{\partial} \\
 &= \sum_{T \in \mathcal{T}[\partial U]} \int_T d\phi_h \wedge * dw_\partial. \nonumber
\end{align}
As discussed above, in the case where the degrees of freedom are the nodal values and the finite-dimensional function space is given by nodal interpolating shape functions, the discrete weak Cartan form reproduces the discrete Cartan form in \citet{MaPaSh1998} and \citet{Ch2008}. However, we will now explicitly show this for this example. We express the action as a function of the components $\phi_h^{ij}$:
$$ S_U[\{\phi^{ij}_h\}] = \int_U [d\phi_h \wedge * d\phi_h - N(\phi_h) dt \wedge dx] = \int_U \left[ \frac{1}{2} \sum_{i,j}\sum_{k,l}\phi^{ij}_h\phi^{kl}_h d\psi_{ij} \wedge * d \psi_{kl} \right]. $$
Let $ij \in \mathcal{I}[\partial U]$, i.e., the index corresponds to a node on $\partial U$, consisting of indices $ij$ such that either $i = 0 \text{ or } M$ or $j = 0 \text{ or } N$. \citet{MaPaSh1998} and \citet{Ch2008} define the discrete Cartan form associated to this node to be
\begin{equation} \label{MPS Discrete Cartan for Wave Eq}
\frac{\partial S_U[\{\phi_h^{kl}\}]}{\partial \phi_h^{ij}} d\phi^{ij}_h,
\end{equation}
where $d$ is the vertical exterior derivative along the fiber and not the exterior derivative on the base space. Compute
$$ \frac{\partial S_U[\{\phi^{kl}_h\}]}{\partial \phi_h^{ij}} = \int_U \left[ \sum_{k,l}\phi^{kl}_h d\psi_{ij}\wedge * d\psi_{kl} \right]. $$
With coordinates $\phi_h^{ij}$ on $\Lambda^0_h$, we can express the vector field $W = \sum_{k,l} W^{kl} \partial/\partial \phi_h^{kl}$ and hence $W^{kl}(\phi_h) = w^{kl}$. Pairing (\ref{MPS Discrete Cartan for Wave Eq}) with $W$ and summing over all $ij \in \mathcal{I}[\partial U]$, we see that this gives (\ref{Discrete Cartan Form for Wave Eq}), since $w_\partial = \sum_{ij \in \mathcal{I}[\partial U]} w^{ij} \psi_{ij}$ and $\psi_{ij}$ for $ij \in \mathcal{I}[\partial U]$ are supported on $\cup_{ T \in \mathcal{T}[\partial U]}T$.

Finally, we now discuss in what sense the discrete weak Cartan form for this example converges to the weak Cartan form. Consider a node $ij \in \mathcal{I}[\partial U]$ along, say, the $\{t = t_0\}$ edge of $\partial U$, so that $i = 0$. We compute part of the discrete Cartan form for a boundary variation $w^{0j}$ associated to this node. Namely, we compute the part associated to the derivative in the $t$ direction, since this is the normal direction along this edge. This is given by 
\begin{align*}
 \int_U \sum_{k,l} \phi_h^{kl} \chi'_k(t) \xi_l(x) w^{0j} \chi'_0(t) \xi_j(x) dt\wedge dx &=  \int_U \sum_{k=0}^1 \sum_{l=j-1}^{j+1} \phi_h^{kl} \chi'_k(t) \xi_l(x) w^{0j} \chi'_0(t) \xi_j(x) dt\wedge dx  \\
 &= \sum_{l=j-1}^{j+1}\frac{\phi_h^{0l} - \phi_h^{1l}}{\Delta t} (\xi_l,\xi_j)_{L^2} w^{0j}.
\end{align*}
Since $(\xi_l,\xi_j)_{L^2}$ for $l=j-1,j,j+1$ has total mass $\Delta x$, this formally converges to $\int \frac{\partial \phi}{\partial n} w \,dx$, where we note that the normal vector on this edge is $-\hat{t}$. Repeating this over all nodes on $\partial U$, the discrete Cartan form formally converges to
$$ \int_{\partial U} \frac{\partial \phi}{\partial n} w\, dl, $$
where $dl$ is the codimension one measure on $\partial U$, which is the weak Cartan form for a solution $\phi$ of the weak Euler--Lagrange equation.  

To be more rigorous about the convergence of the discrete weak Cartan form to the weak Cartan form, we have the bounds
\begin{align*}
\left|\Theta^h_U(\phi_h)\cdot\pi_hv - \Theta_U(\phi)\cdot v\right| &\leq \left|\left(\Theta^h_U(\phi_h) - \Theta^h_U(\pi_h\phi)\right) \cdot \pi_hv\right| \\ &\quad + \Big|\Theta_U(\phi)\cdot (v-\pi_hv)\Big| + \left|\left(\Theta_U(\phi)-\Theta^h_U(\pi_h\phi)\right)\cdot\pi_hv\right| \\
&\leq C |\phi_h - \pi_h\phi|_{H^1} \|\pi_hv\|_{H^1} + C |\phi_h|_{H^1} \|v-\pi_h\phi\|_{H^1} + C|\phi-\pi_h\phi|_{H^1} \|\pi_hv\|_{H^1} \\
&\leq h C(\phi,f),
\end{align*}
where $C(\phi,f)$ is independent of $h$, and we have applied standard estimates for piecewise-linear elements applied to the Poisson equation (see, e.g., \citet{LaTh2003}). Thus, we expect linear convergence of the discrete weak Cartan form to the weak Cartan form.

As a numerical example, we take $U = X = [0,1] \times [0,1]$ with $f(x,y) = -\pi^2 \sin(\pi x) - \pi^2 \sin(\pi y)$, $v(x,y) = e^x + e^y$, and $\Delta t = \Delta x = h$ for various values of $h$. Since we have the analytic solution $\phi(x,y) = \sin(\pi x) + \sin(\pi y)$, we can directly compute the error $\text{E}(h) = \left|\Theta^h_U(\phi_h)\cdot\pi_hv - \Theta_U(\phi)\cdot v\right|$. The linear convergence, i.e.,
$$ \left|\Theta^h_U(\phi_h)\cdot\pi_hv - \Theta_U(\phi)\cdot v\right| \leq \mathcal{O}(h),$$
is shown in Figure \ref{Cartan Error}.

\begin{figure}[h!] 
\includegraphics[scale=0.7]{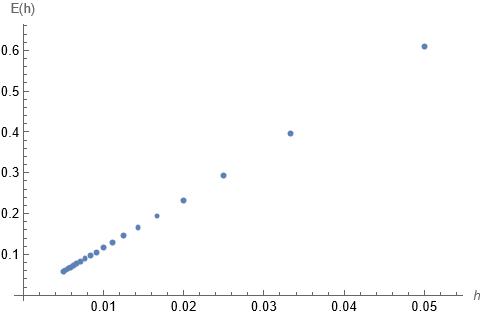}
  \caption{Linear Convergence of the discrete weak Cartan form to the weak Cartan form.}\label{Cartan Error}
\end{figure}

\section{Canonical Semi-discretization of Lagrangian Field Theories}\label{Semi-Discretization Section}
Turning now to the canonical formalism of field theories, we assume that our $(n+1)$-dimensional spacetime $X$ is globally hyperbolic, i.e., $X$ contains a smooth Cauchy hypersurface $\Sigma$ such that every infinite causal curve intersects $\Sigma$ exactly once. It was shown in \citet{BeSa2003} that a globally hyperbolic spacetime is diffeomorphic to the product, $X \cong \mathbb{R} \times \Sigma$. Identifying $X$ with the product, we have a slicing of the spacetime. Taking an interval $I \subseteq \mathbb{R}$, we have the spacelike embeddings
$$i_t: \Sigma \rightarrow X,$$
for each $t \in I$, such that the images $\{ \Sigma_t := i_t(\Sigma) \}_{t\in I}$ form a foliation of $X$. 

We will assume our Lagrangian depends on time-dependent fields as $\mathcal{L}(x^\mu,\varphi,\dot{\varphi},d\varphi)$, where the field $\varphi(t) \in H\Lambda^k(\Sigma_t)$, denoted by $\varphi$ as opposed to the full field $\phi$, and the exterior derivative acts on $\Lambda^k(\Sigma_t)$ for each $t$. 

\begin{remark}\label{Compare Covariant and Canonical Remark}
There is a slight subtlety here when comparing to the covariant theory on the full spacetime $X$. In the covariant theory, we consider $k$-forms on $X$, $\Lambda^kX$, whereas here we are considering $k$-forms on $\Sigma$, $\Lambda^k(\Sigma)$. Letting $\pi_1: \mathbb{R} \times \Sigma \rightarrow \mathbb{R}, \pi_2: \mathbb{R} \times \Sigma \rightarrow \Sigma$ be the projections, we have pointwise,
$$ \wedge^k(T^*X) = \wedge^kT^*(\mathbb{R}\times\Sigma) \cong \Big(\pi_1^*(\wedge^0T^*\mathbb{R}) \wedge \pi_2^*(\wedge^kT^*\Sigma)\Big) \oplus \Big( \pi_1^*(\wedge^1T^*\mathbb{R}) \wedge \pi_2^*(\wedge^{k-1}T^*\Sigma) \Big). $$
This congruence does not hold at the level of sections: to see this in coordinates $(t,x)$ on $\mathbb{R} \times \Sigma$, we have forms which look like $f(t) g(x) dx^{j_1}\wedge\dots\wedge dx^{j_{k}}, f(t) dt \wedge g(x) dx^{j_1}\wedge\dots\wedge dx^{j_{k-1}}$ which cannot give a form which looks like, e.g., $h(t,x) dt \wedge dx^{j_1}\wedge\dots\wedge dx^{j_{k-1}}$ where $h$ is some function that cannot be expressed as a product $f(t)g(x)$. However, we are assuming time-dependent fields $\varphi: t \mapsto  H\Lambda^k(\Sigma)$ so we do have the forms which look like $\varphi(t) = g(t,x)dx^{j_1}\wedge\dots\wedge dx^{j_k}$. Thus, we only need to consider multiple fields to obtain full generality $\varphi_1: t \mapsto H\Lambda^{k}(\Sigma), \varphi_2: t \mapsto H\Lambda^{k-1}(\Sigma)$. Here, we are identifying $\Lambda^1(I) \cong \Lambda^0(I)$, so by $\varphi_2(t)$ we really mean $\varphi_2(t)dt$. Of course, this issue does not arise for scalar functions; however, for $k>0$, one needs to consider multiple fields. 

To be more precise regarding this decomposition, consider first the case $k=0$. Since the exterior derivative on scalar functions on $I \times \Sigma$ splits into $d = d_t + d_\Sigma$, where, in terms of vector field proxies, $d_t = \partial_t, d_\Sigma = \nabla_\Sigma$, one does not need to consider multiple fields. In this case, one has
$$ H\Lambda^0(I \times \Sigma) = H\Lambda^0(I, L^2(\Sigma)) \cap L^2(I, H\Lambda^0(\Sigma)). $$

For the case $k>0$, we first begin with a formal calculation. For any $k$-form  $\phi$ on $I \times \Sigma$, we can express $\phi$ as
$$ \phi = \underbrace{\sum_{I \in I^{k-1}_t} \psi_I(t,x) dt \wedge dx^I}_{\equiv \psi} + \underbrace{\sum_{J \in I^k_\Sigma} \varphi_J(t,x) dx^J}_{\equiv \varphi}, $$
where $I$ and $J$ are multi-indices of size $k-1$ and $k$, respectively, and for a multi-index $I = (i_1,\dots,i_m)$ of size $m$, $dx^I \equiv dx^{i_1} \wedge \dots \wedge dx^{i_m}$. The multi-index set $I^{k-1}_t$ is defined as the set of all multi-indices $(i_1,\dots,i_{k-1})$ such that $i_1 < \dots < i_{k-1}$ and such that each of the indices are non-zero, where we adopt the convention that $dx^0 = dt$. The multi-index set $I^k_\Sigma$ is defined as the set of all multi-indices $(i_1,\dots,i_k)$ such that $i_1 < \dots < i_k$ and such that each of the indices are non-zero. Note that $\varphi$ and $\psi$ are orthogonal with respect to the $L^2\Lambda^k(I \times \Sigma)$ inner product, so square integrability of $\phi$ is equivalent to square integrability of both $\varphi$ and $\psi$. For the square integrability of $d\phi$, we compute the exterior derivative of $\phi$
$$ d\phi = \sum_{I \in I^{k-1}_t} d_\Sigma\psi_I(t,x) \wedge dt \wedge dx^I + \sum_{J \in I^k_\Sigma} \frac{\partial}{\partial t}\varphi_J(t,x) dt \wedge dx^J + \sum_{J \in I^k_\Sigma} d_\Sigma \varphi_J(t,x) \wedge dx^J.  $$
Thus, for square integrability of $\phi$ and $d\phi$, it suffices to have $\psi \in L^2\Lambda^1(I, H\Lambda^{k-1}(\Sigma))$ and $\varphi \in H^1\Lambda^0(I, L^2\Lambda^k(\Sigma)) \cap L^2\Lambda^0(I, H^1\Lambda^k(\Sigma))$. Thus, in the covariant picture, we can view a field $\phi$ as splitting into two fields $\varphi$ and $\psi$. We will treat the case where $\psi = 0$, i.e., we consider theories depending only on $k$-forms of the form
$$ \phi = \sum_{J \in I^k_\Sigma} \varphi_J(t,x) dx^J. $$
In this case, the exterior derivative splits into temporal and spatial derivatives, so we have the identification $(\phi,d\phi) \cong (\varphi,\dot{\varphi}, d_\Sigma \varphi)$. We will subsequently refer to the spatial exterior derivative $d_\Sigma$ simply as $d$.
\end{remark}

We will discuss how a semi-discretization of the variational principle gives rise to finite-dimensional Lagrangian and Hamiltonian dynamical systems (see, for example, \citet{AbMa1978}) and subsequently discuss how the energy-momentum map structure of a canonical field theory (see \citet{GoIsMaMo2004}) is affected by semi-discretization. 


\subsection{Semi-discrete Euler--Lagrange Equations}\label{Semi-Discrete EL Section}

In this section, we formally derive the semi-discrete Euler--Lagrange equations. Given our $\Lambda^{n+1}(X)$-valued Lagrangian density, we can produce an instantaneous density by contracting with the generator of the slicing $\partial/\partial t$, and pulling back by the inclusion of $\Sigma_t$ into $X$. This gives a $\Lambda^n(\Sigma_t)$-valued density, which we will still call $\mathcal{L}$. In coordinates where the density is $L\ dt\wedge V(t)$ and $V(t)$ restricts to a volume form on $\Sigma_t$, $\mathcal{L} = i_t^* L V(t)$.  The action in the canonical framework is given by 
\begin{equation}\label{Canonical Action}
S[\varphi] = \int_I dt \int_{\Sigma_t} \mathcal{L}(x^\mu,\varphi,\dot{\varphi},d\varphi), 
\end{equation}
where $(x^\mu) = (t,x^1,\dots,x^n) = (t,x)$, and $x = (x^i)$ denotes spatial coordinates.

To derive a semi-discrete formulation of the Euler--Lagrange equations, instead of looking at arbitrary variations of the form $v(t,x)$, we instead consider variations of the form $u(t)v(x)$ where $v \in H\Lambda^k(\Sigma)$ and $u \in C_0^2(I,\mathbb{R})$. The basic idea of the semi-discrete formulation is to allow $u$ to be arbitrary but restrict $v$ to a finite-dimensional subspace $\Lambda^k_h$. As in the covariant case, in order to compute the variations formally without going through the Hamilton--Pontryagin principle, we will assume that the projections are cochain projections, with respect to the spatial exterior derivative $d$ on $\Sigma$.
\begin{assumption}
The projections $\pi^m_h: H\Lambda^m(\Sigma) \rightarrow \Lambda^m_h(\Sigma)$ are cochain projections, i.e., $\pi^{k+1}_hd = d\pi^k_h$, with respect to $d: \Lambda^m(\Sigma) \rightarrow \Lambda^{m+1}(\Sigma)$.
\end{assumption}

\begin{remark}Note that we assume a finite element discretization $\Lambda^k_h$ of the fields on the reference space $H\Lambda^k(\Sigma)$, with associated projection $\pi_h$. There are two ways to view the variations with respect to our slicing $\{\Sigma_t\}$. On the one hand, the field variation on the reference space $v \in \Lambda^k_h \subset H\Lambda^k(\Sigma)$ is pulled back to a field variation on a time slice $(i_t^{-1})^*v \in H\Lambda^k(\Sigma_t)$, where we restrict the embedding to its image $i_t: \Sigma \rightarrow \Sigma_t$. On the other hand, we can pull back forms on $\Sigma_t$ to forms on $\Sigma$ via $i^*$, e.g., the Lagrangian density and its derivatives, and perform any relevant integration over the reference space $\Sigma$. We will utilize the latter since in computation it is preferable to work on one reference space. For simplicity, we will not explicitly write the pullbacks $i_t^*$ but rather implicitly incorporate it into the spacetime dependence of the Lagrangian.
\end{remark}

\begin{theorem}\label{Semi-Discrete Euler--Lagrange Theorem}
The semi-discrete Euler--Lagrange equations corresponding to the variational principle $\delta S[\varphi]\cdot (uv) = 0$ for all $v \in \Lambda^k_h$ and $u \in C_0^2(I,\mathbb{R})$ are given by
\begin{equation}\label{Semi-Discrete EL}
\frac{d}{dt}(\partial_3\mathcal{L},v)_{L^2\Lambda^k(\Sigma)} - (\partial_2\mathcal{L},v)_{L^2\Lambda^k(\Sigma)} - (\partial_4\mathcal{L},dv)_{L^2\Lambda^{k+1}(\Sigma)} = 0, \text{ for all } v \in \Lambda^k_h \text{ and } t \in I,
\end{equation}
where $\mathcal{L}$ is evaluated at $(x^\mu,\varphi,\dot{\varphi},d\varphi)$. 
\begin{proof}
With $\mathcal{L}$ evaluated at $(x^\mu,\varphi,\dot{\varphi},d\varphi)$, compute
\begin{align*}
0 &= \delta S[\varphi]\cdot (uv) = \frac{d}{d\epsilon}\Big|_{\epsilon = 0} S[\phi + \epsilon u v] 
\\ &= \int_Idt\int_{\Sigma} \Big[\partial_2\mathcal{L}\wedge\star u(t) v + \partial_3\mathcal{L}\wedge\star \dot{u}(t) v + \partial_4\mathcal{L}\wedge\star u(t) dv\Big]
\\ &= \int_Idt \Big[ \int_{\Sigma} \Big(\partial_3\mathcal{L}\wedge\star v \Big)\dot{u}(t) + \int_{\Sigma} \Big(\partial_2\mathcal{L}\wedge\star  v + \partial_4\mathcal{L}\wedge\star dv\Big)u(t) \Big]
\\ &= \int_Idt \Big[(\partial_3\mathcal{L},v)_{L^2}\dot{u}(t) + (\partial_2\mathcal{L},v)_{L^2}u(t) + (\partial_4\mathcal{L},dv)_{L^2}u(t) \Big]
\\ &= -\int_Idt \Big[\frac{d}{dt}(\partial_3\mathcal{L},v)_{L^2} - (\partial_2\mathcal{L},v)_{L^2} - (\partial_4\mathcal{L},dv)_{L^2} \Big] u(t).
\end{align*}
Since $u \in C_0^2(I,\mathbb{R})$ is arbitrary, the terms in the brackets vanish, which gives (\ref{Semi-Discrete EL}).
\end{proof}
\end{theorem}

\begin{remark}
Similar to our discussion of the covariant case, there is a naturality relation in the variational principle when using spatial cochain projections for the semi-discrete theory. In particular,
$$ S[\pi_h\varphi] = \int_Idt \int_\Sigma \mathcal{L}(x^\mu,\pi_h\varphi,\pi_h \dot{\varphi},d\pi_h\varphi) = \int_Idt \int_\Sigma \mathcal{L}(x^\mu,\pi_h\varphi,\pi_h \dot{\varphi},\pi_h d\varphi) =: S_h[\varphi], $$
so that the restricted variational principle can be realized as a full variational principle on a degenerate action, $\delta S[\pi_h\phi]\cdot (u\  \pi_hv) = \delta S_h[\phi]\cdot (uv)$. Analogous to the discussion in the covariant case, the cochain property additionally removes the ambiguity of how one should discretize the spatial derivative of the field, i.e., whether one should project before or after taking the spatial derivative.
\end{remark}

We now show that the semi-discrete Euler--Lagrange equation (\ref{Semi-Discrete EL}) arises from an instantaneous Lagrangian. To do this, let $\{v_i\}$ be a basis for $\Lambda^k_h$. We define the instantaneous semi-discrete Lagrangian to be
\begin{equation}\label{InstantaneousLagrangian1}
L_h(t,\varphi^i,\dot{\varphi}^i) = \int_\Sigma \mathcal{L}(x^\mu, \varphi^iv_i, \dot{\varphi}^iv_i, \varphi^i dv_i), 
\end{equation}
where $\varphi = \varphi^i(t) v_i \in C^2(I, \Lambda^k_h)$ and the associated action $S_h[\{\varphi^i\}] = \int_I dt L_h(t,\varphi^i,\dot{\varphi}^i)$. We enforce the variational principle over curves $u = u^i(t)v_i \in C^2_0(I,\Lambda^k_h)$. The variational principle yields
\begin{align*}
0 &= dS_h[\{\varphi^i\}]\cdot \{u^j\} = \frac{d}{d\epsilon}\Big|_0 S_h[\{\varphi^i + \epsilon u^j\}] = \sum_j \int_I dt \Big( \frac{\partial L_h}{\partial \varphi^j} (t,\varphi^i,\dot{\varphi}^i) u^j + \frac{\partial L_h}{\partial \dot{\varphi}^j}(t,\varphi^i,\dot{\varphi}^i) \dot{u}^j \Big) \\
&= \sum_j \int_I dt \Big[ \frac{\partial L_h}{\partial \varphi^j} (t,\varphi^i,\dot{\varphi}^i)  - \frac{d}{dt} \frac{\partial L_h}{\partial \dot{\varphi}^j}(t,\varphi^i,\dot{\varphi}^i) \Big] u^j.
\end{align*}
This holds for arbitrary $u^j \in C^2_0(I,\mathbb{R})$, so the term in the brackets,
\begin{equation}\label{Euler--Lagrange Equation}
\frac{\partial L_h}{\partial \varphi^j} (t,\varphi^i,\dot{\varphi}^i)  - \frac{d}{dt} \frac{\partial L_h}{\partial \dot{\varphi}^j}(t,\varphi^i,\dot{\varphi}^i) = 0,
\end{equation}
vanishes for each $j$ by the fundamental lemma of the calculus of variations. Expressing the derivatives of $L_h$ in terms of $\mathcal{L}$,
\begin{subequations}
\begin{align}
\frac{\partial L_h}{\partial \varphi^j} &= (\partial_2\mathcal{L},v_j)_{L^2\Lambda^k(\Sigma)} + (\partial_2\mathcal{L},dv_j)_{L^2\Lambda^{k+1}(\Sigma)}, \label{Semi-Discrete Lagrangian Derivative Equation 1} \\
\frac{\partial L_h}{\partial \dot{\varphi}^j} &= (\partial_3\mathcal{L}, v_j)_{L^2\Lambda^k(\Sigma)}. \label{Semi-Discrete Lagrangian Derivative Equation 2}
\end{align}
\end{subequations}
Substituting these expressions into equation (\ref{Euler--Lagrange Equation}), we see that this is equation (\ref{Semi-Discrete EL}) with the choice $v=v_j$. This holds for each basis form $v_j$ and hence for arbitrary $v \in \Lambda^k_h$. 

We will now introduce a Hamiltonian structure associated with the semi-discretization and show that, in the hyperregular case, this instantaneous Lagrangian system is equivalent to an instantaneous Hamiltonian system.

\subsection{Symplectic Structure of Semi-discrete Dynamics and Hamiltonian Formulation}\label{Semi-discrete Symplectic Structure Section}

Having derived the semi-discrete Euler--Lagrange equation (\ref{Semi-Discrete EL}), we now relate the symplectic structure on the cotangent space of the full field space $T^*H\Lambda^k(\Sigma)$ to a symplectic structure on the discretized space $T^*\Lambda^k_h$, and show that the semi-discrete Euler--Lagrange equations are equivalent to a Hamiltonian flow on $T^* \Lambda^k_h$ if the Lagrangian is hyperregular. 

We work with the reference space $\Sigma$, since via the diffeomorphism $i_t: \Sigma \rightarrow \Sigma_t$, we can pullback forms on $\Sigma$ to $\Sigma_t$ or vice versa, or forms on iterated exterior bundles, such as the symplectic form which is an element of $\Lambda^2(T^*H\Lambda^k(\Sigma))$. On the full phase space $T^*H\Lambda^k(\Sigma)$, the canonical one-form $\theta \in \Lambda^1(T^*H\Lambda^k(\Sigma))$ is given in coordinates by 
\begin{equation}\label{canonical form in canonical formalism}
\theta\big|_{(\varphi,\pi)} = \int_{\Sigma}\pi_A d\varphi^A \otimes d^nx_0,
\end{equation}
and the corresponding symplectic form $\omega = -d\theta$ is given by 
$$ \omega\big|_{(\varphi,\pi)} = \int_{\Sigma} (d\varphi^A \wedge d\pi_A) \otimes d^nx_0. $$
Using the projection map $\pi_h: H\Lambda^k(\Sigma) \rightarrow \Lambda^k_h$, we have the pullback $\pi_h^*: T^*\Lambda^k_h \rightarrow T^*H\Lambda^k(\Sigma)$ and the twice iterated pullback $\pi_h^{**}: \Lambda^p(T^*H\Lambda^k(\Sigma)) \rightarrow \Lambda^p(T^*\Lambda^k_h)$ for any $p$. We define $\theta_h \equiv \pi_h^{**}\theta$ and $\omega_h \equiv \pi_h^{**}\omega = -d\theta_h \in \Lambda^2(T^*\Lambda^k_h)$. To find an expression for $\theta_h$ and $\omega_h$, we will introduce global coordinates on $T^*\Lambda^k_h$. Let $\{v_i\}$ be a finite element basis for $\Lambda^k_h$; we will use the components $\varphi^i$ of the basis expansion $\varphi = \varphi^i v_i$ as the coordinates on $\Lambda^k_h$. Similarly, if we identify $T\Lambda^k_h \cong \Lambda^k_h \times \Lambda^k_h$, then we have a basis for $T^*_\varphi\Lambda^k_h$ consisting of $v^i := (\cdot,v_i)_{L^2}$. This gives the trivialization $T^*\Lambda^k_h \cong \Lambda^k_h \times (\Lambda^k_h)^*$ with global coordinates $(\varphi,\pi) \sim (\varphi^i,\pi_i)$ where $\varphi = \varphi^i v_i$ and $\pi = \pi_i v^i$. We will denote these coordinates using vector notation $\vec{\varphi} = (\varphi^i)$, $\vec{\pi} = (\pi_i)$. 

\begin{prop}
The $1$-form $\theta_h$ is given in the above coordinates by
\begin{equation}\label{Semi-Discrete One Form}
\theta_h = v^j(v_i) \pi_j d\varphi^i = d\vec{\varphi}^TM\vec{\pi},
\end{equation}
where the mass matrix $M$ has components $M_i^{\ j} := v^j(v_i) = \int_\Sigma v_i v_j d^nx_0.$
Furthermore, the $2$-form $\omega_h = -d\theta_h$ is a symplectic form on $T^*\Lambda^k_h$ with coordinate expression
\begin{equation} \label{Semi-Discrete Symplectic Form}
\omega_h = d\varphi^i \wedge v^j(v_i) d\pi_j = d\vec{\varphi}^T\wedge M d\vec{\pi}.
\end{equation}
\begin{proof}
Let $(\varphi,\pi) \in T^*\Lambda^k_h$ and $U \in T_{(\varphi,\pi)}(T^*\Lambda^k_h)$, with coordinate expression
$$ U(\varphi,\pi) = \Phi^i \frac{\partial}{\partial\varphi^i} + \Pi_i\frac{\partial}{\partial \pi_i}. $$
Note that $\theta|_{(\varphi',\pi')}(V)$ gives the canonical pairing between the $\partial/\partial \varphi'$ component of $V$ and $\pi'$ by equation (\ref{canonical form in canonical formalism}). Then, since $\pi_h^*:T^*\Lambda^k_h \hookrightarrow T^*H\Lambda^k(\Sigma)$ is an inclusion, $T\pi_h^*$ is an inclusion on the corresponding tangent space, which gives
\begin{align*}
\theta_h|_{(\varphi,\pi)}(U) = \theta|_{\pi_h^*(\varphi,\pi)} (T\pi_h^*U) = \langle \Phi,\pi\rangle = \Phi^i \pi_j \int_\Sigma v_iv_j d^nx_0 = v^j(v_i)\pi_j \Phi^i = v^j(v_i)\pi_j d\varphi^i(U).
\end{align*}
Equation (\ref{Semi-Discrete Symplectic Form}) then follows from taking (minus) the exterior derivative of equation (\ref{Semi-Discrete One Form}).

The nondegeneracy and closedness of $\omega_h$ clearly follow from the (global) coordinate expression (\ref{Semi-Discrete Symplectic Form}) above. In particular, since the mass matrix $M$ is invertible (hence nondegenerate), $\omega_h$ is nondegenerate. Closedness follows from 
$$d\omega_h = d^2\vec{\varphi}^T \wedge M d\vec{\pi} - d\vec{\varphi}^T \wedge dM \wedge d\vec{\pi} - d\vec{\varphi}^T\wedge M d^2\vec{\pi} = 0.$$ Alternatively, $\omega_h$ is closed as the pullback of a closed form $\omega$. 
\end{proof}
\end{prop}

\begin{remark}
Under a change of basis, $\omega_h$ can be seen as a canonical symplectic form on $T^*\Lambda^k_h$. To see $\omega_h$ in canonical form, we change basis. Let $Q$ be an orthogonal matrix which diagonalizes the symmetric mass matrix $M$, i.e., $QMQ^T = D$. Define coordinates $\vec{q} = Q\vec{\varphi}$ and $\vec{p} = DQ\vec{\pi}$; then
$$ \omega_h = d\vec{\varphi}^T \wedge M d\vec{\pi} = d\vec{\varphi}^T \wedge Q^T DQ d\vec{\pi} = d(Q\vec{\varphi})^T \wedge d(DQ\vec{\pi}) = d\vec{q}^T \wedge d\vec{p}. $$
However, we will work with the form of $\omega_h$ corresponding to the finite element basis (\ref{Semi-Discrete Symplectic Form}) since it is more directly applicable to our discretization. Also, if we chose the dual basis $l^j$ to be different from the basis $v^j = (\cdot,v_j)$, $M$ would not necessarily be symmetric but would still define a symplectic form. This follows from the fact that, for a finite element method to be consistent, one requires that the matrix with components $l^j(v_i)$ is invertible. Hence, it is more natural to work with the coordinates $(\vec{\varphi},\vec{\pi})$.
\end{remark}

Let $H_d: T^*\Lambda^k_h \rightarrow \mathbb{R}$ be a given semi-discrete Hamiltonian, expressed in our global coordinates as $H_d(\vec{\varphi},\vec{\pi})$. Later, we will choose the semi-discrete Hamiltonian induced by the semi-discrete Lagrangian. The dynamics of the Hamiltonian system $(\omega_h,H_d)$ is given by the flow generated by the Hamiltonian vector field $X_{H_d}$ satisfying $X_{H_d} \lrcorner\ \omega_h = dH$, or with vector field components $X_{H_d} = (\dot{\varphi}^i,\dot{\pi}_i)$, 
\begin{equation}\label{Hamiltonian Dynamics for Semi-Discrete Symplectic Form}
\begin{cases} \ M_{\ k}^j\dot{\varphi}^k = \frac{\partial H_d}{\partial \pi_j}, \\
\ M_j^{\ k}\dot{\pi}_k = - \frac{\partial H_d}{\partial \varphi^j}. \end{cases}
\end{equation}
\begin{remark}
In the above, we denote row $j$ and column $k$ of $M$ as $M_j^{\ k}$ and for $M^T$ as $M^j_{\ k}$, which allows for the more general case where $M$ is asymmetric that was discussed previously. If we define $\vec{z}$ as the concatenation of $\vec{\varphi}$ and $\vec{\pi}$, the equations (\ref{Hamiltonian Dynamics for Semi-Discrete Symplectic Form}) can be written in skew-symmetric form,
$$ \frac{d}{dt}\vec{z} = J_M \nabla_{\vec{z}} H_d, $$
where $J_M = \begin{pmatrix} 0 & (M^{-1})^T \\ -M^{-1} & 0 \end{pmatrix} $.
\end{remark}

\begin{remark}
In our discussion of the covariant discretization of Lagrangian field theories, we saw that the variation of the discretized action on the discrete space can be naturally related to the variation of a degenerate action on the full space.  In the semi-discrete setting, an analogous statement can be made in terms of the semi-discrete symplectic structure and a presymplectic structure on the full space. Namely, we have the symplectic form $\omega_h \in \Lambda^2(T^*\Lambda^k_h)$. Now, consider the presymplectic form $\tilde{\omega}_h \in \Lambda^2(T^*H\Lambda^k)$ defined by $\tilde{\omega}_h = i_h^{**}\omega_h$ where $i_h = (\pi_h)^{\dagger}: \Lambda^k_h \hookrightarrow H\Lambda^k$ is the inclusion.  Clearly, $\tilde{\omega}_h$ is closed as the pullback of a closed form. To see that it is degenerate, observe that for any $V,W \in \mathfrak{X}(T^*H\Lambda^k)$, we have,
$$ \tilde{\omega}_h(V,W) = (i_h^{**} \pi_h^{**}\omega)(V,W) = \omega(T(\pi_h^* i_h^*)V, T(\pi_h^* i_h^*)W). $$
Since $i_h\pi_h$ has a nontrivial kernel, so does $T(\pi_h^*i_h^*) = T(i_h\pi_h)^*$ and hence $\tilde{\omega}_h$ is degenerate. The flow of a vector field in the kernel of $\tilde{\omega}_h$, projected back to the semi-discrete space, corresponds to equivalent states in the semi-discrete setting. Quotienting the presymplectic manifold $(T^*H\Lambda^k, \tilde{\omega}_h)$ by the orbits of the flow of vector fields in the kernel of $\tilde{\omega}_h$ gives the symplectic manifold $(T^*\Lambda^k_h,\omega_h)$. This relates a symplectic flow on $(T^*\Lambda^k_h,\omega_h)$ to an equivalence class of presymplectic flows on $(T^*H\Lambda^k,\tilde{\omega}_h)$, where the equivalence class is formed by orbits of the flow of vector fields in the kernel of $\tilde{\omega}_h$.
\end{remark}

We also allow our semi-discrete Hamiltonian to explicitly depend on time, $H_d: I \times T^*\Lambda^k_h \rightarrow \mathbb{R}$, i.e., the domain of $H_d$ is the extended phase space $I \times T^*\Lambda^k_h$. The dynamics are now given by any vector field $X_{H_d}$ on the extended phase space such that $X_{H_d}\lrcorner\ (\omega_h + dH_d\wedge dt) = 0$, where $\omega_h$ is extended to the full phase space by pulling back along the projection $I \times T^*\Lambda^k_h \rightarrow T^*\Lambda^k_h$. If we consider the vertical component $X^V_{H_d}$ of $X_{H_d}$ with respect to the trivial bundle $I \times \Lambda^k_h \rightarrow I$, then the above is equivalent to $X^V_{H_d} \lrcorner\ \omega_h = d_v{H_d}$ holding for all times. This is given again by equation (\ref{Hamiltonian Dynamics for Semi-Discrete Symplectic Form}) but with explicit time dependence in $H_d$. Here,  $d_vH_d$ is the vertical exterior derivative of $H_d$ with coordinate expression $d_vH_d(t,\varphi,\pi) = \frac{\partial H_d}{\partial \varphi^i}d\varphi^i + \frac{\partial H_d}{\partial \pi_j} d\pi_j$. We could also allow explicit time dependence in $M$, but since we pullback our integration to $\Sigma$, we view $M$ as constant and absorb the time dependence into $H_d$. 

Now, we would like to relate the semi-discrete Euler--Lagrange equations (\ref{Semi-Discrete EL}) to the Hamiltonian dynamics of $\omega_h$ by making a particular choice of semi-discrete Hamiltonian. The first step is to produce a Hamiltonian associated to the instantaneous Lagrangian 
$$ L(t,\varphi,\dot{\varphi})= \int_{\Sigma}\mathcal{L}(x^\mu,\varphi,\dot{\varphi},d\varphi). $$
To do this, we use the Legendre transform, which takes the form $\pi = \partial L/\partial \dot{\varphi}$. The pairing of $\pi$ with a tangent vector field with components $(\varphi,v)$ is given by computing the variation 
$$ \langle \pi, v \rangle = \left\langle \frac{\partial L}{\partial \dot{\varphi}}, v \right\rangle = \frac{d}{d\epsilon}\Big|_{\epsilon = 0}L(t,\varphi,\dot{\varphi}+\epsilon v) = (\partial_3\mathcal{L},v)_{L^2\Lambda^k}. $$
The instantaneous Hamiltonian is given by
$$ H(t,\varphi,\pi) = \langle \pi,\dot{\varphi} \rangle - L(t,\varphi,\dot{\varphi}), $$ 
where the $\dot{\varphi}$ dependence is removed either by extremizing over $\dot{\varphi}$ or, assuming $L$ is hyperregular, by inverting the Legendre transform to obtain $\dot{\varphi}$ as a function of $(\varphi,\pi)$. Restricting to our finite element space $T^*\Lambda^k_h$ gives a semi-discrete Hamiltonian $H_h$, defined by
$$ H_h(t,\varphi^i,\pi_i) = H(t,\varphi^iv_i,\pi_iv^j) = \langle \pi_jv^j,\dot{\varphi}^iv_i\rangle - L(t,\varphi^iv_i,\dot{\varphi}^iv_i) = M_i^{\ j}\pi_j\dot{\varphi}^i - L(t,\varphi^iv_i,\dot{\varphi}^iv_i). $$
Note that $H_h$ corresponds to the Legendre transform of the semi-discrete Lagrangian (\ref{InstantaneousLagrangian1}), where we recall the duality pairing between $(\varphi^j,\pi_j)\in T^*\Lambda^k_h$ and $(\varphi^i,\dot{\varphi}^i)\in T\Lambda^k_h$ is given by $M_i^{\ j}\pi_j\dot{\varphi}^i$.
\begin{prop}\label{Equivalence of Semi-Discrete Formulations Prop}
Assume that $L_h$ is hyperregular, then the dynamics associated with the Hamiltonian system $(\omega_h,H_h)$ is equivalent to the semi-discrete Euler--Lagrange equations (\ref{Semi-Discrete EL}). 
\begin{proof}
Since we assumed that $L_h$ is hyperregular, i.e., that the associated Legendre transform is a diffeomorphism $T\Lambda^k_h \rightarrow T^*\Lambda^k_h$, we have $\dot{\varphi}^i$ as a function of $(\varphi^j,\pi_j)$. To verify the equivalence, we compute the equations (\ref{Hamiltonian Dynamics for Semi-Discrete Symplectic Form}) for our given system. Compute for $L$ evaluated at $(t,\varphi^iv_i,\dot{\varphi}^iv_i)$, 
\begin{align*}
M_j^{\ k}\dot{\pi}_k &= -\frac{\partial H_h}{\partial \varphi^j} = -\frac{\partial}{\partial\varphi^j}\Big(M_i^{\ k}\pi_k\dot{\varphi}^i - L \Big)
\\ &= -M_i^{\ k}\pi_k\frac{\partial\dot{\varphi}^i}{\partial\varphi^j} +  \frac{\partial}{\partial\varphi^j}\int_{\Sigma_t}\mathcal{L}(x^\mu,\varphi^iv_i,\dot{\varphi}^iv_i,\varphi^idv_i) 
\\ &= -M_i^{\ k}\pi_k\frac{\partial\dot{\varphi}^i}{\partial\varphi^j} + \int_{\Sigma_t}\Big[\partial_2\mathcal{L}\wedge\star\frac{\partial(\varphi^iv_i)}{\partial\varphi^j} + \partial_3\mathcal{L}\wedge\star\frac{\partial(\dot{\varphi}^iv_i)}{\partial\varphi^j} + \partial_4\mathcal{L}\wedge\star\frac{\partial(\varphi^idv_i)}{\partial\varphi^j}\Big]
\\ &= -M_i^{\ k}\pi_k\frac{\partial\dot{\varphi}^i}{\partial\varphi^j} + \int_{\Sigma_t}\Big[\partial_2\mathcal{L}\wedge\star v_j + \partial_3\mathcal{L}\wedge\star v_i \frac{\partial\dot{\varphi}^i}{\partial\varphi^j} + \partial_4\mathcal{L}\wedge\star dv_j\Big]
\\ &= -M_i^{\ k}\pi_k\frac{\partial\dot{\varphi}^i}{\partial\varphi^j} + (\partial_3\mathcal{L},v_i)\frac{\partial\dot{\varphi}^i}{\partial\varphi^j} + (\partial_2\mathcal{L},v_j)_{L^2} + (\partial_4\mathcal{L},dv_j)_{L^2}
\\ &= (\partial_2\mathcal{L},v_j)_{L^2} + (\partial_4\mathcal{L},dv_j)_{L^2},
\end{align*}
where in the second to last line, the first two terms cancel since $(\partial_3\mathcal{L},v_i) = \langle\pi,v_i\rangle = \langle\pi_kv^k,v_i\rangle = M_i^{\ k}\pi_k$. Then, note the left hand side is equivalently given by
$$ M_j^{\ k}\dot{\pi}_k = M_j^{\ k}\frac{d}{dt}\pi_k = \frac{d}{dt}\big(M_j^{\ k}\pi_k\big) = \frac{d}{dt}\big( \langle v^k,v_j\rangle\pi_k \big) =  \frac{d}{dt} \langle \pi,v_j\rangle = \frac{d}{dt}(\partial_3\mathcal{L},v_j)_{L^2}. $$
Thus, 
$$ \frac{d}{dt}(\partial_3\mathcal{L},v_j)_{L^2} = (\partial_2\mathcal{L},v_j)_{L^2} + (\partial_4\mathcal{L},dv_j)_{L^2}, $$
which holds for each $j$ and hence is equivalent to (\ref{Semi-Discrete EL}). 
\end{proof}
\end{prop}
\begin{remark}
In the above proposition, we assumed that $L_h$ was hyperregular for the equivalence. If $L_h$ is not hyperregular, corresponding to a degenerate field theory, the dynamics associated to $H_h$ evolve over a primary constraint surface. In this case, the dynamics of $H_h$ on the constraint surface corresponds to a (not necessarily unique) solution of the semi-discrete Euler--Lagrange equation. In this setting, the dynamics are associated to the modified Hamiltonian $\bar{H}(\vec{\varphi},\vec{\pi},\lambda) = H(\vec{\varphi},\vec{\pi})+ \lambda^A\Phi_A(\vec{\varphi},\vec{\pi})$. 

The above also shows that, in the hyperregular case, the semi-discrete Euler--Lagrange equations correspond to a symplectic flow. The associated symplectic form is the pullback of $\omega_h$ by the Legendre transform $\mathbb{F}L_h: T\Lambda^k_h \rightarrow T^*\Lambda^k_h$. In the non-regular case, the semi-discrete Euler--Lagrange equations correspond to a presymplectic flow.
\end{remark}

To summarize, in this section, we have pulled back the symplectic structure on $T^*H\Lambda^k$ to $T^*\Lambda^k_h$ and showed that the dynamics of the Hamiltonian system $(\omega_h,H_h)$ is equivalent, in the hyperregular case, to the semi-discrete Euler--Lagrange equations of the corresponding Lagrangian system. By applying a numerical integrator for the finite-dimensional Hamiltonian system associated to $H_h$, we obtain a full discretization of the evolution problem for a field theory. 

\subsection{Energy-Momentum Map}\label{Energy-Momentum Section}
In this section, we examine how symmetries in the canonical formulation are affected by the semi-discretization of the field theory. In the canonical setting, the manifestation of the covariant momentum map is the energy-momentum map. If a vector in the Lie algebra of the symmetry group gives rise to an infinitesimal generator on $X$ which is transverse to the foliation, its pairing with the energy-momentum map equals the instantaneous Hamiltonian defined by that generator (the ``energy'' component). On the other hand, if the corresponding generator is tangent to the foliation, the pairing is given by the usual momentum map of the instantaneous Hamiltonian theory, corresponding to the canonical form (\ref{canonical form in canonical formalism}) (the ``momentum'' component). We will see that, in the case of an equivariant discretization, the iterated pullback of the energy-momentum map provides the natural energy-momentum structure of the semi-discrete theory.

We start by investigating the momentum map structure of the semi-discrete theory. Let $K$ be a Lie group acting on $H\Lambda^k$, with Lie algebra $\mathfrak{k} := T_e K$. For $\eta \in K$, we denote the group action $\overline{\eta}\varphi := \eta\cdot\varphi$ and the associated cotangent action is given by $\widetilde{\eta}:=(\overline{\eta^{-1}})^*$. We use the same notation for these actions restricted to $\Lambda^k_h$ and $T^*\Lambda^k_h$, where the restriction is well-defined if the projection is group-equivariant.

\begin{prop}\label{Momentum Map Prop}
Assume that $K$ acts by symplectomorphisms on $(T^*H\Lambda^k,\omega)$. Since $K$ acts by cotangent lifts on $T^*H\Lambda^k$, it admits a canonical momentum map $J: T^*H\Lambda^k \rightarrow \mathfrak{k}^*$. Furthermore, assume that the projection map $\pi_h$ is equivariant with respect to the $K$-action on $H\Lambda^k$ and $\Lambda^k_h$, i.e., $\pi_h\bar{\eta} \varphi = \bar{\eta}\pi_h\varphi$. Then, $K$ acts by cotangent-lifted symplectomorphisms on $(T^*\Lambda^k_h,\omega_h)$ and the canonical momentum map for this action $J_h$ is given by $J_h = \pi_h^{**}J = J \circ \pi_h^*$.
\begin{proof}
To see that $K$ preserves $\omega_h$, for any $\eta \in K$, by equivariance, we have that
$$ \widetilde\eta^*\omega_h = (\overline{\eta^{-1}})^{**}\pi_h^{**}\omega = (\overline{\eta^{-1}}\pi_h)^{**}\omega = (\pi_h\overline{\eta^{-1}})^{**}\omega = \pi_h^{**}(\overline{\eta^{-1}})^{**}\omega = \pi_h^{**}\omega = \omega_h.$$
A similar result holds for $\theta_h$, since $K$ preserves $\theta$ by virtue of the fact that it acts by cotangent lifted actions.

The canonical momentum map $J$ is given by $\langle J(\varphi,\pi),\xi\rangle = \xi_{T^*H\Lambda^k}(\varphi,\pi)\lrcorner\, \theta|_{(\varphi,\pi)}$ for $(\varphi,\pi) \in T^*H\Lambda^k$ whereas $\langle J_h(\varphi,\pi),\xi \rangle = \xi_{T^*\Lambda^k_h}(\varphi,\pi)\lrcorner\, \theta_h|_{(\varphi,\pi)}$ for $(\varphi,\pi) \in T^*\Lambda^k_h$. These are both momentum maps for their respective actions since $K$ acts by cotangent lifts. Then,
\begin{align*}
\langle J_h(\varphi,\pi),\xi \rangle &= \xi_{T^*\Lambda^k_h}(\varphi,\pi)\lrcorner\, \theta_h = \xi_{T^*\Lambda^k_h}(\varphi,\pi)\lrcorner\, \pi_h^{**}\theta 
\\ &= [T\pi_h^*\xi_{T^*\Lambda^k_h}(\varphi,\pi)]\lrcorner\, \theta
= \Big[T\pi_h^*\frac{d}{dt}\Big|_{t = 0}\widetilde{e^{t\xi}}(\varphi,\pi)\Big] \lrcorner\, \theta  
\\ &= \Big[\frac{d}{dt}\Big|_{t = 0}\pi_h^*(\overline{e^{-t\xi}})^*(\varphi,\pi)\Big] \lrcorner\, \theta  
\\ &= \Big[\frac{d}{dt}\Big|_{t = 0}(\overline{e^{-t\xi}}\pi_h)^*(\varphi,\pi)\Big] \lrcorner\, \theta  
\\ &= \Big[\frac{d}{dt}\Big|_{t = 0}(\pi_h\overline{e^{-t\xi}})^*(\varphi,\pi)\Big] \lrcorner\, \theta  
\\ &= \Big[\frac{d}{dt}\Big|_{t = 0}(\overline{e^{-t\xi}})^*\pi_h^*(\varphi,\pi)\Big] \lrcorner\, \theta  
\\ &= \xi_{T^*H\Lambda^k}(\pi_h^*(\varphi,\pi))\lrcorner\, \theta = \langle (J\circ\pi_h^*)(\varphi,\pi),\xi \rangle.
\end{align*}
where we have implicitly evaluated $\theta_h$ at $(\varphi,\pi)$ and $\theta$ at $\pi_h^*(\varphi,\pi)$. Hence, $J_h = J \circ \pi_h^*$ or, equivalently, $J_h = \pi_h^{**}J$.
\end{proof}
\end{prop}

\begin{remark}
As can be seen in the proof, one does not need full $K$-equivariance of the projection, but only infinitesimal equivariance, i.e., $\pi_h(e^{t\xi} \varphi) - e^{t\xi}\pi_h\varphi = o(t)$. 

Furthermore, one can weaken the notion of equivariance to $\pi_h \bar{\eta} = \overline{ \psi_h(\eta) } \pi_h$, where $\psi_h: K \rightarrow K$ is a Lie group homomorphism. In this case, if $\tilde{\psi}_h$ denotes the induced Lie algebra homomorphism, we can see from the above proof that the semi-discrete momentum map is related to the original momentum map via $\langle J_h ,\xi\rangle = \langle J \circ \pi_h^*, \tilde{\psi}_h(\xi)\rangle.$

As discussed in the covariant case, the weakening of this condition can allow us to construct more general projections.
\end{remark}

\begin{corollary}
Assuming as in the proposition, if $J$ is $\text{Ad}^*$-equivariant, then so is $J_h$.
\begin{proof}
This follows immediately from $J_h = J \circ \pi_h^*$, $K$-equivariance of $\pi_h$, and the $\text{Ad}^*$-equivariance $J \circ \widetilde{\eta} = \text{Ad}^*_\eta J$ (where $\text{Ad}^*_\eta := (\text{Ad}(\eta^{-1}))^*)$:
\begin{align*} 
J_h \circ \widetilde{\eta} &= J \circ \pi_h^* \circ (\overline{\eta^{-1}})^* =  J \circ (\overline{\eta^{-1}})^* \circ \pi_h^*
\\ &= J \circ \widetilde{\eta} \circ \pi_h^* = (\text{Ad}^*_\eta J) \circ \pi_h^* = \text{Ad}^*_\eta (J \circ \pi_h^*) = \text{Ad}^*_\eta J_h,
\end{align*}
where the equality $(\text{Ad}^*_\eta J) \circ \pi_h^* = \text{Ad}^*_\eta (J \circ \pi_h^*)$ holds since the coadjoint action acts on $J$ after it is evaluated on its input, which is then an element of $\mathfrak{k}^*$. In particular, $(\text{Ad}_\eta^*J)(\varphi,\pi) := \text{Ad}_\eta^* ( J(\varphi,\pi) )$, so that 
$$ ((\text{Ad}^*_\eta J) \circ \pi_h^* )(\varphi,\pi) = (\text{Ad}^*_\eta J) (\pi_h^*(\varphi,\pi)) = \text{Ad}^*_\eta (J (\pi_h^*(\varphi,\pi))) =   \text{Ad}^*_\eta ((J \circ \pi_h^*)(\varphi,\pi))). $$
Stated another way, this follows from associativity of the composition of functions, viewing $\text{Ad}_\eta^*$ as a function $\mathfrak{k}^* \rightarrow \mathfrak{k}^*$.
\end{proof}
\end{corollary}

\begin{remark}
Of course, since $K$ acts by cotangent lifts and hence by canonical symplectomorphisms, $J$ is an $\text{Ad}^*$-equivariant momentum map, and the corollary tells us that $J_h$ is as well. However, as we remark below, one may consider more general actions which admit momentum maps, and it is not necessarily the case that those momentum maps are $Ad^*$-equivariant. The result of the previous corollary still holds in this more general setting.
\end{remark}

The naturality of the momentum map structures from the previous proposition and corollary can be summarized via the following commuting diagram; for any $\eta \in K$,

\centerline{\xymatrix{
T^*H\Lambda^k \ar[dr]^{\widetilde{\eta}} \ar[dddrr]_{J} & & & & \ar[llll]^(0.483){\pi_h^*} T^*\Lambda^k_h \ar[ld]_{\widetilde{\eta}} \ar[llddd]^{J_h}
\\& T^*H\Lambda^k \ar[dr]_{J} & & \ar[ll]^{\pi_h^*} T^*\Lambda^k_h \ar[dl]^{J_h} &
\\ &  & \mathfrak{k}^* & &
\\ &  & \mathfrak{k}^* \ar[u]^(0.6){\text{Ad}^*_\eta} & &.
}}

\begin{remark} In the above proposition, we only assumed that $\pi_h$ was equivariant with respect to the $K$-action on the configuration space, and it follows that $\pi_h^*$ is equivariant with respect to the lifted action on the cotangent space. However, for more general actions on the cotangent space that do not arise from a cotangent lift, one must instead assume $\pi_h^*$ is equivariant with respect to this action. In this case, if the $K$-action on $T^*H\Lambda^k$ admits a momentum map $J$, then $J_h = \pi_h^{**}J$ is a momentum map for the action on $T^*\Lambda^k_h$. To verify this, let $(\varphi,\pi) \in T^*\Lambda^k_h$. We know that $d\langle J,\xi\rangle = i_{\xi_{T^*H\Lambda^k}}\omega$. Thus, 
$$d\langle J_h,\xi\rangle = \pi_h^{**}d\langle J,\xi\rangle = \pi_h^{**}(i_{\xi_{T^*H\Lambda^k}}\omega). $$
Then, observe that by equivariance, $\xi_{T^*H\Lambda^k}(\varphi,\pi) = T\pi^*_{h}\xi_{T^*\Lambda^k_h}(\varphi,\pi)$. Then, for any $X \in TT^*_{(\varphi,\pi)}\Lambda^k_h,$
$$ d\langle J_h,\xi\rangle (X) = (i_{\xi_{T^*H\Lambda^k}}\omega)(T\pi^*_{h}X) = \omega(T\pi^*_{h}\xi_{T^*\Lambda^k_h},T\pi^*_{h}X) = (\pi_h^{**}\omega)(\xi_{T^*\Lambda^k_h},X) = (i_{\xi_{T^*\Lambda^k_h}}\omega_h)(X), $$
where the above is evaluated at $(\varphi,\pi)$, which verifies that $J_h$ is a momentum map. For the subsequent discussion, we will assume that $K$ acts by cotangent lifts.
\end{remark}

We now define the energy-momentum map (\citet{GoIsMaMo2004}) and its semi-discrete counterpart. We consider vectors on $\Sigma_t$ with both tangent components in $T\Sigma_t$ and components transverse to the foliation, which in our adapted coordinates are in the span of $\partial/\partial t$. We extend the canonical form $\theta$ to act on vector fields on the extended phase space in the same way that we extended $\omega_h$ in our previous discussion of time-dependence. Let $\widetilde{\mathcal{L}}$ denote the Lagrangian density on the full spacetime, which is related to the instantaneous Lagrangian density by $\mathcal{L} = i_t^*\partial_t\lrcorner\tilde{\mathcal{L}}$. Define the map $\mathcal{J}$ from $I \times T^*H\Lambda^k$ to the dual of the space of vector fields on the extended phase space, via
\begin{equation}\label{Energy-Momentum Map vector fields}
\langle \mathfrak{J}(t,\varphi,\pi), V\rangle = (V\lrcorner\, \theta)(t,\varphi,\pi) - \int_{\Sigma}i_t^*V_t\lrcorner\, \widetilde{\mathcal{L}}(x^\mu,\varphi,\dot{\varphi},d\varphi),
\end{equation}
where we view $\dot{\varphi}$ as a function of $(\varphi,\pi)$, and where $V_t$ is the tangent-lift of the bundle projection $I \times T^*H\Lambda^k(\Sigma_t) \rightarrow I$ applied to $V$. 
\begin{prop}\label{Energy-Momentum Map Prop}
$\mathfrak{J}$ is the energy-momentum map, in the following sense:
\begin{enumerate}[label=(\roman*)]
\item \textbf{(Energy)} Let $\Phi^H_t$ denote the Hamiltonian flow of $H$, and $X_H$ be the associated generator on the extended phase space, then,
$$\langle \mathfrak{J}(t,\varphi,\pi), X_H\rangle =  H(t,\varphi,\pi).$$ 
\item \textbf{(Momentum)} If $V$ is tangent to the foliation, then,
$$\langle \mathfrak{J}(t,\varphi,\pi),V\rangle = (V\lrcorner\, \theta)(t,\varphi,\pi),$$
and in particular, if there is a $K$-action as in Proposition (\ref{Momentum Map Prop}) on the phase space over $\Sigma_t$, its momentum map $J$ is given by
$$ \langle J(t,\varphi,\pi),\xi\rangle = \langle \mathfrak{J}(t,\varphi,\pi),\xi_{T^*H\Lambda^k} \rangle,$$
such that, for each fixed $t$, $d\langle J(t,\varphi,\pi),\xi\rangle = \xi_{T^*H\Lambda^k}\lrcorner\, \omega(t,\varphi,\pi)$. 
\end{enumerate}
\begin{proof}
For the proof of (i), in local coordinates, we have that
$$ X_H = \frac{d}{dt}\Big|_{t=0} \Phi_t^H(t',\varphi,\pi) = \frac{\partial}{\partial t} + \dot{\varphi}^A \frac{\partial}{\partial\varphi^A} + \dot{\pi}_B\frac{\partial}{\partial\pi_B}, $$
and $(X_H)_t = \partial/\partial t$. Using expressions (\ref{canonical form in canonical formalism}) and (\ref{Energy-Momentum Map vector fields}), and the definition of the instantaneous Lagrangian density $\mathcal{L} = i_t^*\partial_t\lrcorner\, \widetilde{\mathcal{L}}$, we have that
\begin{align*}
\langle \mathfrak{J}(t,\varphi,\pi),X_H \rangle &= (X_H \lrcorner\, \theta)(t,\varphi,\pi) - \int_{\Sigma} i_t^*(X_H)_t\lrcorner\, \widetilde{\mathcal{L}}(x^\mu,\varphi,\dot{\varphi},d\varphi)
\\ &= \dot{\varphi}^A\frac{\partial}{\partial\varphi^A} \lrcorner\, \Big(\int_{\Sigma_t} \pi_Ad\varphi^A\otimes d^nx_0\Big) - \int_\Sigma i_t^*\frac{\partial}{\partial t}\lrcorner\, \widetilde{\mathcal{L}}(x^\mu,\varphi,\dot{\varphi},d\varphi) 
\\ &= \int_{\Sigma_t} \pi_A\dot{\varphi}^Ad^nx_0 - \int_\Sigma\mathcal{L}(t,x^i,\varphi,\dot{\varphi},d\varphi) 
\\ &= \langle \pi,\dot{\varphi} \rangle - L(t,\varphi,\dot{\varphi}) = H(t,\varphi,\pi). 
\end{align*} 

For the proof of (ii), note that for $V$ tangent to the foliation, $V_t=0$, which immediately gives the first equation of (ii). Setting the vector field to an infinitesimal generator of a $K$-action gives the momentum map
$$ \langle \mathfrak{J}(t,\varphi,\pi),\xi_{T^*H\Lambda^k} \rangle = (\xi_{T^*H\Lambda^k} \lrcorner\, \theta)(t,\varphi,\pi). $$
\end{proof}
\end{prop}

We now define the semi-discrete analogue of the energy-momentum map (\ref{Energy-Momentum Map vector fields}). Define the semi-discrete energy-momentum map $\mathcal{J}_h$ from $I \times T^*\Lambda^k_h$ to the dual of vector fields on the extended discrete phase space, via
\begin{equation}\label{Discrete Energy-Momentum Map vector fields}
\langle \mathfrak{J}_h(t,\varphi,\pi), V\rangle = (V\lrcorner\, \theta_h)(t,\varphi,\pi) - \int_{\Sigma}i_t^*V_{t,h}\lrcorner\, \widetilde{\mathcal{L}}_h(x^\mu,\varphi,\dot{\varphi},d\varphi),
\end{equation}
where $V_{t,h} = (T\pi_h^*V)_t$ and $\widetilde{\mathcal{L}}_h$ is the restriction of $\widetilde{\mathcal{L}}$ via precomposition with $\pi_h^*$. Of course, the analogous statement of the previous proposition holds for the semi-discrete energy-momentum map. Furthermore, $\mathcal{J}_h$ is the restriction of $\mathcal{J}$ in the following sense. 

\begin{prop}\label{Discrete EM Map as Restriction}
For $(t,\varphi,\pi)$ in the extended discrete phase space and $V$ a vector field over this space,
$$ \langle \mathfrak{J}_h(t,\varphi,\pi), V \rangle = \langle \mathfrak{J}(t,\pi_h^*(\varphi,\pi)), T\pi_h^* V \rangle. $$
\begin{proof}
This follows directly from the definitions,
\begin{align*}
\langle \mathfrak{J}(t,\pi_h^*(\varphi,\pi)), T\pi_h^* V \rangle &= (T\pi_h^*V \lrcorner\, \theta)(t,\pi_h^*(\varphi,\pi)) - \int_{\Sigma}i_t^{*}(T\pi_h^*V)_t \lrcorner\, \widetilde{\mathcal{L}}(t,\pi_h^*[(\varphi,\dot{\varphi},d\varphi)|_{(\varphi,\pi)}])
\\ &= (V \lrcorner\, \pi_h^{**}\theta)(t,\varphi,\pi) - \int_{\Sigma}i_t^*V_{t,h}\lrcorner\, \widetilde{\mathcal{L}}_h(t,\varphi,\dot{\varphi},d\varphi)
\\ &= (V\lrcorner\, \theta_h)(t,\varphi,\pi) - \int_{\Sigma}i_t^*V_{t,h}\lrcorner\, \widetilde{\mathcal{L}}_h(t,\varphi,\dot{\varphi},d\varphi)
 = \langle \mathfrak{J}_h(t,\varphi,\pi),V \rangle.
\end{align*}
\end{proof}
\end{prop}
The significance of this definition of the semi-discrete energy-momentum map is that it recovers the properties of Proposition \ref{Energy-Momentum Map Prop} in the semi-discrete setting. 
\begin{prop} 
\
\begin{enumerate}[label=(\roman*)]
\item \textbf{(Semi-discrete Energy)}
For $(t,\varphi,\pi)$ in the extended discrete phase space,
$$ \langle \mathfrak{J}_h(t,\varphi,\pi),X_{H_h} \rangle = H_h(t,\varphi,\pi). $$
\item \textbf{(Semi-discrete Momentum)}
If there is a $K$-action on the discrete phase space, then the momentum map $J_h$ is given by
$$ \langle J_h(t,\varphi,\pi),\xi\rangle = \langle \mathfrak{J}_h(t,\varphi,\pi),\xi_{T^*\Lambda^k_h}\rangle. $$
Furthermore, if the $K$-action on the discrete space arises from an action on the full space such that $\pi_h$ is $K$-equivariant, then for any $\xi \in \mathfrak{k}$, 
$$ \langle \mathfrak{J}_h(t,\varphi,\pi),\xi_{T^*\Lambda^k_h} \rangle = \langle \mathfrak{J}(t,\pi_h^*(\varphi,\pi)),\xi_{T^*H\Lambda^k}\rangle . $$
\end{enumerate}
\begin{proof}
The first two equations follow from analogous computations to the proof of Proposition \ref{Energy-Momentum Map Prop}. The last equation follows from the equivariance of $\pi_h$,
$$ T\pi_h^* \xi_{T^*\Lambda^k_h}(\varphi,\pi) = \xi_{T^*H\Lambda^k}(\pi_h^*(\varphi,\pi)),$$
and Proposition \ref{Discrete EM Map as Restriction}.
\end{proof}
\end{prop}

The significance of a semi-discrete analogue of the energy-momentum map, aside from extending the semi-discrete momentum map structure, that was discussed in Proposition \ref{Momentum Map Prop}, is in determining semi-discrete analogues of Noether's second theorem, which we will pursue in subsequent work. 

\subsection{Temporal Discretization of the Semi-Discrete Theory}\label{Tensor Product Discretization Section}
To complete the discussion of the semi-discrete theory, we must of course discretize in time. We obtain a full discretization of the semi-discrete theory by discretizing the semi-discrete Euler--Lagrange equation (\ref{Semi-Discrete EL}) in time via a Galerkin Lagrangian variational integrator applied to the instantaneous semi-discrete Lagrangian (\ref{InstantaneousLagrangian1}), and show that this is equivalent to the full spacetime DEL (\ref{DEL 1a}) with tensor product elements. The associated finite element on the full spacetime is a tensor product mesh, obtained by discretizing the space $\Sigma$ and extending these elements in time by a partition of $I$. Of course, this is not the most general setup for a spacetime discretization, but often one wishes to discretize in time separately. For example, by choosing the appropriate temporal basis functions, the computation becomes local in time so that one can time march the solution from the initial data, instead of solving the entire DEL on the spacetime grid. Furthermore, there are constructions of cochain projections for tensor product elements (\citet{Ar2018}) so that with these finite element spaces, the naturality of the variational principle discussed in Section 2 carries over in the tensor product setting. 

Assume the same setup as in the discussion of the semi-discrete theory. Furthermore, assume that we have a finite element discretization of $H_0(I)$, the space of square integrable functions in time with square integrable derivative, which vanish on $\partial I$, with basis functions $\{w_\alpha \}$. Recall the instantaneous semi-discrete Lagrangian (\ref{InstantaneousLagrangian1}) is a function of the curves $\varphi^i(t),\dot{\varphi}^i(t)$ which are the coefficients of the expansions of $\varphi(t),\dot{\varphi}(t) \in \Lambda^k_h$ relative to the basis $\{v_i\}$ of $\Lambda^k_h$. Using the basis $\{w_\alpha\}$, we discretize these curves as 
\begin{align*}
\varphi^i(t) = (\varphi^i)^{\alpha} w_\alpha(t),
\end{align*}
where $\varphi(t) = (\varphi^i)^\alpha w_\alpha(t)v_i \in \Lambda^k_h$ in this notation. We consider the associated fully discrete action as a function of the coefficients,
$$ S[\{(\varphi^i)^{\alpha}\}] = \int_I dt L_h(t,\varphi^i(t), \dot{\varphi}^i(t)) = \int_I dt L_h(t,(\varphi^i)^{\alpha} w_\alpha, (\varphi^i)^{\alpha} \dot{w}_\alpha). $$
Enforcing the discrete variational principle in time gives the weak form of the Euler--Lagrange equations,
$$ 0 = \frac{\delta S}{\delta (\varphi^i)^\alpha} = \left(\frac{\partial L_h}{\partial \varphi^i}, w_\alpha\right)_{L^2(I)} + \left(\frac{\partial L_h}{\partial \dot{\varphi}^i},\dot{w}_\alpha\right)_{L^2(I)}. $$
Substituting equations (\ref{Semi-Discrete Lagrangian Derivative Equation 1}) and (\ref{Semi-Discrete Lagrangian Derivative Equation 2}) gives
\begin{align*}
0 &=  -((\partial_3\mathcal{L},v_i)_{L^2\Lambda^k(\Sigma)}, \dot{w}_\alpha)_{L^2(I)} - ((\partial_2\mathcal{L},v_i)_{L^2\Lambda^k(\Sigma)},w_\alpha)_{L^2(I)} - ((\partial_4\mathcal{L},dv_i)_{L^2\Lambda^{k+1}(\Sigma)}, w_\alpha)_{L^2(I)} \\ 
&= -(\partial_3\mathcal{L}, v_i \dot{w}_\alpha)_{L^2(I,L^2\Lambda^k(\Sigma))} - (\partial_2\mathcal{L},v_iw_\alpha)_{L^2(I,L^2\Lambda^k(\Sigma))} - (\partial_4\mathcal{L},(dv_i) w_\alpha)_{L^2(I,L^2\Lambda^{k+1}(\Sigma)) } \\
&= -\left(\frac{\partial \mathcal{L}}{\partial\dot{\varphi}}, v_i \dot{w}_\alpha\right)_{L^2(I,L^2\Lambda^k(\Sigma))} - \left(\frac{\partial \mathcal{L}}{\partial \varphi},v_iw_\alpha\right)_{L^2(I,L^2\Lambda^k(\Sigma))} - \left(\frac{\partial \mathcal{L}}{\partial (d\varphi)},(dv_i) w_\alpha\right)_{L^2(I,L^2\Lambda^{k+1}(\Sigma))}.
\end{align*}
Note that these equations can also be obtained directly from the semi-discrete Euler--Lagrange equations (\ref{Semi-Discrete EL}) by applying the Galerkin method in time with respect to the basis $\{w_\alpha\}$. Here, $d$ denotes the spatial exterior derivative on $\Sigma$. If $d_t$ denotes the temporal exterior derivative and we identity functions on $I$ with one-forms on $I$, we have $\dot{w}_\alpha \cong d_t w_\alpha$. If $d_T = d + d_t$ denotes the total exterior derivative on $\Sigma \times I$, then $d_T(v_i w_\alpha) = (dv_i)w_\alpha + v_i d_t w_\alpha$, where, as discussed in Remark \ref{Compare Covariant and Canonical Remark}, we are considering $k$-forms of the form $\sum_{J\in I^k_\Sigma} \phi_J(t,x)dx^J$. We now view the time-dependent $k$-form $\varphi: t \mapsto \varphi(t)$ as a $k$-form $\phi$ on spacetime, so the above can be written as
\begin{align*}
0 &= - \left(\frac{\partial \mathcal{L}}{\partial \phi},v_iw_\alpha\right)_{L^2\Lambda^k(\Sigma \times I)} - \left(\frac{\partial \mathcal{L}}{\partial (d\phi)},(dv_i) w_\alpha\right)_{L^2(I,L^2\Lambda^{k+1}(\Sigma))}  -\left(\frac{\partial \mathcal{L}}{\partial(d_t\phi)}, v_i d_tw_\alpha\right)_{L^2\Lambda^1(I,L^2\Lambda^{k}(\Sigma))} \\
&= - \left(\frac{\partial \mathcal{L}}{\partial \phi},v_iw_\alpha\right)_{L^2(\Sigma \times I)} - \left(\frac{\partial \mathcal{L}}{\partial (d_T\phi)},d_T(v_iw_\alpha)\right)_{L^2\Lambda^{k+1}(\Sigma \times I) },
\end{align*}
which is the DEL (\ref{DEL 1a}) with tensor product basis $\{v_iw_\alpha\}$. 

Note that this result can also be obtained from the semi-discrete Hamiltonian setting, assuming that $L_h$ is hyperregular, using the fact that the semi-discrete Hamiltonian and semi-discrete Lagrangian formulations are equivalent by Proposition \ref{Equivalence of Semi-Discrete Formulations Prop}, and the fact that Galerkin Lagrangian variational integrators and Galerkin Hamiltonian variational integrators are equivalent in the hyperregular case, as established in \citet{LeZh2009}.

\section{Conclusion and Future Directions}
In this paper, we showed how discretizing the variational principle for Lagrangian field theories using finite element cochain projections naturally gives rise to a discrete variational structure which is analogous to the continuum variational structure. Namely, the discrete variational structure is encoded by the discrete Cartan form. Our discrete Cartan form generalizes the discrete Cartan form introduced by \citet{MaPaSh1998} to more general finite element spaces within the finite element exterior calculus framework. Using the discrete Cartan form, we expressed a discrete multisymplectic form formula and a discrete Noether theorem in direct analogy to their continuum counterparts. Furthermore, we studied semi-discretization of Lagrangian PDEs by spatial cochain projections, showing that such semi-discretization gives rise to semi-discrete symplectic, Hamiltonian, and energy-momentum map structures. Finally, we related the methods obtained by covariant discretization and canonical semi-discretization in the case of tensor product finite elements. 

In the paper, we outlined several possible research directions, including studying particular field theories and showing rigorous convergence of the discrete Cartan form, constructing group-equivariant cochain projections, and establishing a discrete Noether's second theorem utilizing the semi-discrete energy-momentum map. Another natural research direction would be to extend the discrete variational structures presented here to the discontinuous Galerkin setting and compare them with the results obtained in the multisymplectic Hamiltonian setting by  \citet{McAr2020}. In particular, we expect that in this setting, the discrete Cartan form would only involve integration over $\partial U$, since boundary variations can be localized to codimension-one simplices, unlike for conforming finite element spaces. Furthermore, we aim to investigate how the discrete variational structures presented in this paper in the conforming setting, and extended to the discontinuous Galerkin setting, can be used to provide a geometric variational framework for studying lattice field theories, building on the discrete variational framework for lattice field theories initiated in \citet{ArZa2014}. 

\section*{Acknowledgements}
BT was supported by the NSF Graduate Research Fellowship DGE-2038238, and by NSF under grants DMS-1411792, DMS-1813635. ML was supported by NSF under grants DMS-1411792, DMS-1345013, DMS-1813635, by AFOSR under grant FA9550-18-1-0288, and by the DoD under grant 13106725 (Newton Award for Transformative Ideas during the COVID-19 Pandemic).

\nocite{*}

\bibliographystyle{plainnat}
\bibliography{fieldtheory.bib}

%
%
%
%

\end{document}